\documentclass[12pt,a4paper]{amsart}
\usepackage[centering,text={15cm,24cm}]{geometry}
\usepackage{amsmath, amssymb, amsthm, stmaryrd}
\usepackage[utf8]{inputenc}
\usepackage{alltt}
\usepackage{marginnote}
\usepackage{hyperref}
\usepackage{amscd}
\usepackage{url}
\usepackage{setspace}
\usepackage{xcolor}
\usepackage{tikz}
\usepackage{float}
\usepackage{subfloat}
\usepackage{paralist}
\usepackage{comment}
\usepackage{wrapfig}
\usepackage[all]{xy}

\usepackage{tikz-cd}
\usepackage{graphicx}
\usepackage{color}
\usepackage{mathrsfs}
\usepackage{mathtools}
\usepackage{marvosym}
\usepackage{srcltx}
\usepackage{enumitem}
\usepackage{lscape}
\usepackage{bbm} 
\usepackage{bm} 
\usepackage{upgreek} 
\usepackage{dcpic}
\usepackage[numbers]{natbib}

\usepackage{amsmath} 

\definecolor{shadecolor}{rgb}{1,0.9,0.7}

\newtheorem{theorem}{Theorem}[section]
\newtheorem{lemma}[theorem]{Lemma}
\newtheorem{lemma-definition}[theorem]{Lemma-Definition}
\newtheorem{proposition}[theorem]{Proposition}
\newtheorem{corollary}[theorem]{Corollary}

\newtheorem{claim}{Claim}

\theoremstyle{definition}

\newtheorem{definition}[theorem]{Definition}

\newtheorem{example}[theorem]{Example}

\newtheorem*{acknowledgement}{Acknowledgement}
\newtheorem*{conventions}{Conventions}

\theoremstyle{remark}
\newtheorem{remark}[theorem]{Remark}

\numberwithin{equation}{section}
\numberwithin{figure}{section}

\newcommand {\lfor} {\llbracket}
\newcommand {\rfor} {\rrbracket}

\newcommand{\DD} {\mathbb{D}}
\newcommand{\HH} {\mathbb{H}}
\newcommand{\NN} {\mathbb{N}}
\newcommand{\ZZ} {\mathbb{Z}}
\newcommand{\QQ} {\mathbb{Q}}
\newcommand{\RR} {\mathbb{R}}
\newcommand{\CC} {\mathbb{C}}
\newcommand{\FF} {\mathbb{F}}
\newcommand{\PP} {\mathbb{P}}
\renewcommand{\AA} {\mathbb{A}}

\newcommand {\shA}  {\mathcal{A}}

\newcommand {\shB}  {\mathcal{B}}

\newcommand {\shD}  {\mathcal{D}}

\newcommand {\shE}  {\mathcal{E}}
\newcommand {\shF}  {\mathcal{F}}
\newcommand {\shG}  {\mathcal{G}}
\newcommand {\shH}  {\mathcal{H}}
\newcommand {\shHom} {\mathcal{H}\!\text{\textit{om}}}

\newcommand {\shL}  {\mathcal{L}}
\newcommand {\shM}  {\mathcal{M}}

\newcommand {\shN}  {\mathcal{N}}
\newcommand {\shO}  {\mathcal{O}}

\newcommand {\shS}  {\mathcal{S}}
\newcommand {\shT}  {\mathcal{T}}

\newcommand {\shP}  {\mathcal{P}}
\newcommand {\shV}  {\mathcal{V}}

\newcommand {\shX}  {\mathcal{X}}

\newcommand{\codim}[2]{\mathrm{codim}(#1, #2)}

\newcommand {\coker} {\operatorname{coker}}

\newcommand {\dlog} {\operatorname{dlog}}

\newcommand {\eps}  {\varepsilon}

\newcommand {\gp}  {{\operatorname{gp}}}

\newcommand {\Hom}  {\operatorname{Hom}}

\newcommand {\hra} {\hookrightarrow}

\newcommand {\id}  {\operatorname{id}}

\newcommand {\invneg}  {\mathrel{\llcorner}}

\renewcommand {\ker } {\operatorname{ker}}
\newcommand {\kk} {\Bbbk}

\newcommand {\liminv} {\varprojlim}

\newcommand {\lra}  {\longrightarrow}

\newcommand {\M} {\mathcal{M}}

\renewcommand {\max} {{\operatorname{max}}}

\newcommand {\ol} {\overline}
\newcommand {\one} {\mathbbm{1}}
\let \op \operatorname

\renewcommand{\P}  {\mathscr{P}}

\newcommand {\ra}  {\to}

\newcommand {\red}[1]{{\color{red} #1}}

\newcommand {\rk} {\operatorname{rk}}

\newcommand {\shLS} {\mathcal{LS}}

\newcommand {\sing} {\mathrm{sing}}

\newcommand {\Spec} {\operatorname{Spec}}

\newcommand {\sra} {\twoheadrightarrow}

\newcommand  {\todo}[1]{{\marginpar{\tiny #1}}}

\newcommand {\Tor}  {\operatorname{Tor}}

\newcommand {\X} {\mathfrak X}
\newcommand {\Y} {\mathfrak Y}

\newcommand{\cO}{\mathcal{O}}

\renewcommand{\log}{\operatorname{log}}

\renewcommand{\P}{\mathscr{P}}

\renewcommand{\log}{\operatorname{log}}
\newcommand{\cM}{\mathcal{M}}
\newcommand{\cE}{\mathcal{E}}
\newcommand{\al}{\alpha}
\newcommand{\cK}{\mathcal{K}}

\newcommand{\cI}{\mathcal{I}}

\newcommand{\F}{\mathcal{F}}
\newcommand{\cL}{\mathcal{L}}
\newcommand{\cS}{\mathcal{S}}
\newcommand{\E}{\mathcal{E}}

\newcommand{\K}{\mathcal{K}}
\newcommand{\cP}{\mathcal{P}}

\newcommand{\cT}{\mathcal{T}}

\newcommand{\dime}{\mathrm{dim}}
\newcommand{\cH}{\mathcal{H}}
\newcommand{\G}{\mathcal{G}}
\newcommand{\C}{\mathcal{C}}

\newcommand{\oS}{{{}^0\!S}}
\newcommand{\oX}{{{}^0\!X}}
\newcommand{\kS}{{{}^k\!S}}
\newcommand{\kX}{{{}^k\!X}}
\newcommand{\kV}{{{}^k\!V}}
\newcommand{\PV}{\operatorname{PV}}

\def\mydate{\ifcase\month \or January\or February\or March\or
April\or May\or June\or July\or August\or September\or October\or 
November\or December\fi \space\number\day,\space\number\year}

\begin{document}


\title
[Smoothing toroidal crossing spaces]
{Smoothing toroidal crossing spaces}
\author{Simon Felten, Matej Filip, Helge Ruddat}

\begin{abstract}
We prove the existence of a smoothing for a toroidal crossing space under mild assumptions. By linking log structures with infinitesimal deformations, the result receives a very compact form for normal crossing spaces. The main approach is to study log structures that are incoherent on a subspace of codimension two and prove a Hodge--de Rham degeneration theorem for such log spaces which also settles a conjecture by Danilov. 
We show that the homotopy equivalence between Maurer--Cartan solutions and deformations combined with Batalin--Vilkovisky theory can be used to obtain smoothings. The construction of new Calabi--Yau and Fano manifolds as well as Frobenius manifold structures on moduli spaces are potential applications.
\end{abstract}

\address{JGU Mainz, Institut f\"ur Mathematik, Staudingerweg 9, 55128 Mainz, Germany}
\email{sfelten@uni-mainz.de}
\email{ruddat@uni-mainz.de}

\address{University of Ljubljana, Institute of Mathematics, Physics and Mechanics, Trzaska cesta 25, Slovenia}
\email{matej.filip@fe.uni-lj.si}

\address{Universit\"at Hamburg, Fachbereich Mathematik, Bundesstraße 55, 20146 Hamburg, Germany}
\email{helge.ruddat@uni-hamburg.de}

\thanks{This work was supported by DFG research grant RU 1629/4-1 and the Carl Zeiss foundation.}

\maketitle
\setcounter{tocdepth}{1}
\tableofcontents

\section{Introduction}
For two smooth components $Y_1,Y_2$ meeting in a smooth divisor $D$ a folkloristic statement says that a necessary condition for $X=Y_1\cup Y_2$ to have a smoothing is that the two normal bundles are dual to each other, i.e., $\shN_{D/Y_1}\otimes\shN_{D/Y_2}\cong\shO_D$. This statement is actually incorrect. It is true only with the further requirement that the total space of the smoothing be itself smooth. 
Conceptually, $\shN_{D/Y_1}\otimes\shN_{D/Y_2}\cong \shE xt^1(\Omega_X,\shO_X)=:\shT^1_X$ and Friedman famously coined the notion of \emph{d-semistability} which is saying $\shT^1_X\cong\shO_D$ \cite{fri}. 
We are going to generalize the situation by only requiring $\shT^1_X$ to be \emph{generated by global sections} (and beyond). 
For a choice of $s\in\Gamma(D,\shT^1_X)$, the total space of the smoothing will then be of the local form $xy=tf$ where $t$ is the deformation parameter, $Y_1=V(x), Y_2=V(y)$ and $f$ represents $s$ in a local trivialization of $\shT^1_X$. The total space of the smoothing has singularities precisely along $s=0$. 
The local form $xy=tf$ has been found to be abundant in mirror symmetry applications \cite{CKYZ,GrossSiebertI,GrossSiebertII,gs,CLL12,GKR,AAK,RS19}.

We work more generally with a \emph{normal crossing space}, that is, a connected variety $X$ over $\CC$ \'etale locally of the form $z_1\cdot...\cdot z_k=0$ for varying $k\le \dim X+1$. We call a flat map $\shX\ra\DD$ for $\DD$ a holomorphic disk a \emph{smoothing} of $X$ if the central fiber is isomorphic to $X$ and the general fiber is smooth. 
If a smoothing exists, then we call $X$ \emph{smoothable}. 
We say that a normal crossing space has \emph{effective anti-canonical class} if the dual of its dualizing sheaf $\omega_X$ can be represented by a reduced divisor $E$ that meets the strata of $X$ transversely, that is, \'etale locally along $E$,  $X$ is equivalent to $E\times\AA^1$.
We prove the following theorem. 
\begin{theorem} \label{maintheorem-nc}
Let $X$ be a proper normal crossing space with effective anti-canonical class.
If $\shT^1_X$ is generated by global sections and $X_\sing$ is projective, then $X$ is smoothable.
\end{theorem}
The only purpose of the projectivity condition is to apply Bertini's theorem to have available a ``nice'' section of the line bundle $\shT^1_X$ on $X_\sing$. 
Both the projectivity assumption as well as the global generatedness assumption on $\shT^1_X$ can thus be removed if there exists a \emph{schön} section of $\shT^1_X$, that is, a section whose vanishing locus $Z$ is reduced and $X_\sing$ is locally along $Z$ equivalent to $Z\times\AA^1$.
We also prove a more general theorem for toroidal crossing spaces that we give down below (Theorem~\ref{maintheorem-tc}). Theorem~\ref{maintheorem-nc} provides a lot more flexibility than existing smoothing results, notably Friedman's \cite{fri} for surfaces, Kawamata--Namikawa's \cite{KawamataNamikawa1994} for d-semistable Calabi--Yaus and Gross--Siebert's \cite{gs} allowing a singular total space but with much stronger requirements on $X$, see also \cite{tzi,NYot,NYot2,Sano,Lee21}.
\begin{example}
The union $X$ of $d$ hyperplanes in general position in $\PP^n$ is smoothable to a degree $d$ hypersurface, but none of the existing results is able to predict the smoothability of $X$ abstractly.
Indeed, the total space of the smoothing is singular since it requires blowing up the base locus of the smoothing pencil. On the other hand, $\shT^1_X$ is generated by global sections. Theorem~\ref{maintheorem-nc} predicts the smoothability if $d\le n+1$.
\end{example}

\begin{example} 
The simplest type of normal crossing space is one with two smoothly intersecting components:
let $Y$ be a smooth Fano manifold with $-K_Y$ very ample, let $D$ be a smooth section of $-K_Y$ and $X$ be the normal crossing space obtained by identifying two copies of $Y$ along $D$. 
Then $\shT^1_X\cong\shN^{\otimes 2}_{D/Y}$ is generated by global sections and $X$ is Calabi--Yau, so Theorem~\ref{maintheorem-nc} provides a smoothing of $X$.
For Fano threefolds $Y$ that are complete intersections in products of weighted projective spaces the smoothing gives Calabi--Yau threefolds of Euler numbers $-106, -122, -138, -156,$ $-128, -156, -176, -256, -260, -296$. While double intersection situations can be birationally modified to be tractable by the smoothing result in \cite{KawamataNamikawa1994}, this is no longer true for triple (and higher) intersection situations \cite{Lee2018}, but Theorem~\ref{maintheorem-nc} provides smoothings.
\end{example}
Theorem~\ref{maintheorem-nc} considerably facilitates the construction of Calabi--Yau and Fano manifolds. 
Our work generalizes the Gross--Siebert program toward allowing non-toric components in the central fiber as well as more flexibility in the local structure, cf.~Example~\ref{exampleGS}.
We generalize Tziolas's smoothing result for Fanos by dropping its cohomological condition \cite{tzi}.
While we work with toric local models, non-toric deformations of toric local models have applications for smoothing singular toric Fanos or the construction of versal deformations of non-isolated Gorenstein singularities, see.~\cite{CFP,coa2}.
For Whitney umbrella local models, the $\shT^1$-sheaf has recently been computed in \cite{Fantechi2020}.
If the pushforward of the sheaf of differentials from the log smooth locus can be verified to commute with base change for other types of local models, then our smoothing result extends to such situations.

Our results enable the construction of versal Calabi--Yau families and conjecturally a logarithmic Frobenius manifold structure in a formal neighborhood of the extended moduli space, see \cite{bar-kon}, \cite[Theorem~1.3]{CLM}. 
Theorem~\ref{rel-degen} below can be viewed as the statement that the Hodge bundles extend trivially over boundary divisors in the moduli space that have toroidal families above them, see also \cite{LSC}.
Since the smoothing deformations are constructed via the Batalin--Vilkovisky formalism in the Gerstenhaber algebra of (log) polyvector fields \S\ref{sec-maurer-Cartan}, the connection to homological mirror symmetry can be made via  \cite{bar-kon},\cite{kkp}.

\subsection{Method of Proof}
The first step toward proving Theorem~\ref{maintheorem-nc} is to furnish $X$ with a log structure, an idea already found in \cite{KawamataNamikawa1994,gs}. We build a connection between these two works.
A sheaf of sets $\shL\shS_X$ on $X$ classifying log smooth structures locally on $X$ over the standard log point $S$ has been defined and studied in \cite{GrossSiebertI}.
We show in \S\ref{sec-log-infini} there is a canonical map $\shL\shS_X\ra\shT^1_X$ with the property that a section $s\in\Gamma(X_\sing,\shT^1_X)$ yields a log smooth structure on $U:=X\setminus V(s)$, i.e., we obtain a log smooth morphisms $U\ra S$. 
The complement $Z:=V(s)$ has codimension two in $X$. Using Bertini's theorem with the projectivity of $X_\sing$, we can assume that $Z$ is schön as defined above.

In the fashion of Zariski--Steenbrink--Danilov, we consider the differential forms $W^k_{X/S}:=j_*\Omega^k_{U/S}$ for $j:U\hra X$ the inclusion. 
In the logarithmic context, these complexes were defined and studied independently by \cite{Nakayama2010} and \cite{GrossSiebertII}.
A key ingredient for the smoothing of $X$ is the knowledge that the Hodge--de Rham spectral sequence for $W^\bullet_{X/S}$ degenerates at $E_1$, a very hard problem. 
We use the close control over $W^k_{X/S}$\label{define-W} along $Z$ which we gain by using~\cite{GrossSiebertII,rud} to obtain a particular type of \emph{elementary log toroidal} local models for the log structure near $Z$. 
For the proof of the Hodge--de Rham degeneration, we follow the approach by Deligne--Illusie \cite{Deligne1987}: spreading out to finite characteristic and using the Cartier isomorphism. 
However, the lack of coherence poses serious new challenges. 
The hardest technical part is to show the sheaves $W^\bullet_{X/S}$ commute with base change because $j_*$ and $\otimes$ do not commute in general. 
Base change may fail for low characteristics by Example~\ref{baChaViolation}.
However, if the characteristic of the base field is sufficiently large we prove by explicit computation in the elementary log toroidal local models that the sheaves $W^\bullet_{X/S}$ commute with base change.
As a second difficulty, underived pushforward $j_*$ does not ordinarily pass to the derived category and we find a workaround here.
We settle a conjecture by Danilov \cite[15.9]{Danilov1978} along the way (Theorem~\ref{danilov-thm} below).

To show the unobstructedness of log deformations of $X$, we use recent advancements of the Bogomolov--Tian--Todorov theory motivated by the study of mirror symmetry, starting with \cite{kkp} and \cite{bar-kon} which got cultivated to work in algebraic geometry by \cite{iac-man}.
All these works however produce equisingular deformations (because they are intended for deforming smooth spaces). 
The crucial difference to our setup is that while we prescribe local deformations by the log structure, these are not locally trivial deformations.
Building on \cite{FMM12}, recently this difficulty in the theory has been addressed in \cite{CLM,CM} which adapts perfectly to our situation to produce a formal deformation in the prescribed local models, see \S\ref{section-smoothing}. We found the framework of Gerstenhaber algebras to be the most effective to think about the theory which governs our way of parsing \cite{CLM} in \S\ref{sec-maurer-Cartan}, see also \cite{SFel}.
At this point, the assumption about effectiveness of $\omega_X^{-1}$ enters the proof, so that one obtains an isomorphism of $W^\bullet_{X/S}(\log E)$ with the Gerstenhaber algebra of log polyvector fields $\PV^\bullet$ and has the Batalin--Vilkovisky operator $\Delta$ available by transporting the de Rham differential to $\PV^\bullet$ which is used in \S\ref{MC-from-BV}.

To improve the resulting formal smoothing to an analytic smoothing, we use the Grauert--Douady space and Artin approximation as already done in \cite{RS19}. 

\subsection{Toroidal Pairs and Danilov's Conjecture}
A \emph{toroidal pair} $(X,D)$ is a variety $X$ over a field $\kk$ of characteristic zero with Weil divisor $D\subset X$ such that $X$ is \'etale locally equivalent to an affine toric variety with $D$ identified with a reduced toric divisor (not necessarily the entire toric boundary). Danilov defined the sheaf of differentials $\tilde\Omega^p_X(\log D)$ as the push-forward of the usual K\"ahler differentials $\Omega^p_{X_{reg}}(\log D|_{X_{reg}})$ with log poles from the locus $X_{reg} \subset X$ where the space is regular.
\begin{theorem}[Danilov's conjecture] \label{danilov-thm} 
Given a proper toroidal pair $(X,D)$, the Hodge--de Rham spectral sequence
$$E_1^{p,q}=H^q(X,\tilde\Omega^p_X(\log D))\Rightarrow \HH^{p+q}(X,\tilde\Omega^\bullet_X(\log D))$$
degenerates at $E_1$.
\end{theorem}
Special cases of this theorem were known before: 
when $X$ has at worst orbifold singularities \cite{Steenbrink1976}, for $D=\emptyset$ \cite{Blickle2001}, and for $D$ locally the entire toric boundary \cite{Tsuji1999,Illusie2002a}.
We believe that our methods can be extended to prove generalizations of the Akizuki--Nakano--Kodaira vanishing theorem.

\subsection{Toroidal Crossing Spaces, their Log Structures, and Orbifold Smoothings}
If $V=\Spec\kk[P]$ is an affine toric variety given by some toric monoid $P$, consider the map of sheaves $a:\underline P\ra\shO_V, p\mapsto z^p,$ with $\underline P$ denoting the constant sheaf. 
We obtain a sheaf of monoids $\shP_V=\underline P/a^{-1}(\shO_V^\times)$. 
Now $V$ is Gorenstein if and only if the toric boundary $D$ in $V$ is a Cartier divisor, hence given as the zero locus of a monomial $\one\in P$.
\begin{definition}[Siebert,\,Schr\"oer\,\cite{sch-sie}]
A \emph{toroidal crossing space} is an algebraic space $X$ over $\kk$ together with a sheaf of monoids $\shP$ with global section $\one\in\Gamma(X,\shP)$ such that for every point $x\in X$, \'etale locally at $x$, $X$ permits a smooth map to the toric boundary $D_x$ in $V_x=\Spec \kk[\shP_x]$ so that $\shP$ is isomorphic to the pullback of $\shP_{V_x}$ and mapping $\one_x$ to the monomial in $\shP_x$ whose divisor is $D_x$.
\end{definition}

A toroidal crossing space $X$ is automatically Gorenstein, we denote its dualizing line bundle by $\omega_X$. 
The boundary divisor in a Gorenstein toric variety is naturally a toroidal crossing space. 
General hyperplane sections of projective toroidal crossing spaces are again naturally toroidal crossing spaces.
\begin{lemma} 
A normal crossing space is naturally a toroidal crossing space by setting $\shP_x:=\NN^k$ and $\one_x=(1,1,...,1)\in\NN^k$ whenever $X$ is locally at $x$ given by $z_1\cdot...\cdot z_k=0$. 
(While there are other possibilities to turn a normal crossing space into a toroidal crossing space, we will always refer to this one.)
\end{lemma}
The class of toroidal crossing spaces is closed under forming products (but not so the class of normal crossing spaces).
The sheaf $\shP$ provides what Gross and Siebert call a \emph{ghost structure} for $X$ (\cite[Definition~3.16]{GrossSiebertI}), an ingredient to define the sheaf $\shLS_X$ (\cite[Definition~3.19]{GrossSiebertI}) whose sections are in bijection with log structures on $X$ together with a log smooth map to the standard log point $S$. 
By \cite{GrossSiebertI}, $\shLS_X$ embeds into the coherent sheaf $\bigoplus_C j_{C,*}\shN_{\tilde C}$ where the sum is over the irreducible components $C$ of $X_\sing$, $j_C: \tilde C\ra C\ra X$ is the composition of normalization and closed embedding, and $\shN_{\tilde C}$ is a line bundle on $\tilde C$. 
The sheaf $\shLS_X$ often does not have global sections. 
It suffices however to give a section $s$ of $\shLS_X$ on a dense open set $U$ that contains the generic points of the minimal strata of $X$ so that each component $s_C\in \Gamma(U\cap C,\shN_C)$ of $s$ extends to a section of $\shN_C$ on all of $C$ by acquiring simple zeros. 
The zeros define a reduced Cartier divisor $Z_{\tilde C}$ for each $\tilde C$. 
Set $Z=\bigcup_C j_C(Z_{\tilde C}) \subset X$.
The construction of local models along $Z$ in \cite{GrossSiebertII} was generalized in \cite{rud}: locally the coherent log structure, given by $s$ on $U$, extends to an incoherent log structure on $X$ that is still given by certain toric local models, namely from a divisor in an affine toric variety that is not the entire toric boundary, e.g.~like in the definition of toroidal pair above.
A section $s$ of $\shLS_X$ on a dense open set $U$ will be called \emph{simple} if it extends to $X$ by simple zeros and the resulting $Z_{\tilde C}$ satisfy the simpleness~criterion in \S\ref{sec-log-TC}.
Our most general smoothing result is the following.
\begin{theorem} \label{maintheorem-tc}
Let $X$ be a proper toroidal crossing space with a simple~section $s$ of $\shLS_X$ on a dense open set $U$. 
Assume that $\omega^{-1}_X$ permits a section whose divisor of zeros $E$ meets all strata of $X$ and $Z$ transversely (e.g.~when $\omega^{-1}_X\cong\shO_X$, $E=\emptyset$), then $X$ is smoothable to an orbifold with terminal singularities.
\end{theorem}
There is a precise derivation of the types of singularities of the orbifold smoothing from knowing $\shP$ and $Z$, e.g.~for a normal crossing space there will be no singularities in the general fiber of the smoothing and thus combined with the Bertini argument and linking $\shLS_X$ with $\shT^1_X$, we find that Theorem~\ref{maintheorem-tc} implies Theorem~\ref{maintheorem-nc}, see Proposition~\ref{prop-tc-implies-nc}. 
A section of $\shLS_X$ is of complete intersection type (c.i.t.)~as defined in \cite{rud}, roughly speaking, if the log singular set satisfies a transversality assumption. A c.i.t.~section gives rise to a log toroidal morphism. Theorem~\ref{locally-unique-defos} does not hold for the general c.i.t.~case but we obtain the following.

\begin{example} \label{exampleGS}
We follow \cite{GrossSiebertI}. Let $(B,\P,\varphi)$ be a closed oriented tropical manifold with singular locus combinatorially c.i.t.; then the associated space $X_0(B,\shP,s)$ with its vanilla gluing data and log structure satisfies the assumptions of Theorem~\ref{maintheorem-tc} for $E=\emptyset$ if the orbifold nearby fiber models are terminal (given by elementary simplices). 
Smoothings for such spaces had been constructed in \cite{gs} under the stronger assumption of local rigidity, e.g.~the quintic threefold degeneration in $\PP^4$ is not locally rigid but c.i.t..
\end{example}

\subsection{The Hodge--de Rham Spectral Sequence} 
We refer to \cite{kkatoFI,Kato1996,GrossSiebertI,LoAG18} for basic notions of log geometry.
Let $f:X\ra S$ be a log toroidal family as defined in Definition~\ref{toroidDef} below. 
A toroidal pair $(X,D)$ yields an example by giving $X$ the divisorial log structure from $D$ and making $S$ the log trivial point.
The families $X$ over the standard log point featured in Theorem~\ref{maintheorem-tc} give further examples. 
Also, a saturated relatively log smooth morphism $f:X\ra S$ in the sense of \cite{Nakayama2010} is an example.
The complex $W^\bullet_{X/S}$ (see page~\pageref{define-W}) gives rise to a spectral sequence 
$$E(X/S): E^{pq}_1 = R^qf_*W^p_{X/S} \Rightarrow R^{p + q}f_*W^\bullet_{X/S}.$$
Let $Q$ be a sharp toric monoid and $\kk$ be a field of characteristic zero. We prove the following theorems.

\begin{theorem}\label{absDegen} 
Let $S=\Spec (Q\ra\kk)$ and $f: X \to S$ be a proper log toroidal family (with respect to $S \to A_Q$). 
 Then $E(X/S)$ degenerates at $E_1$.
\end{theorem}

Theorem~\ref{absDegen} implies Theorem~\ref{danilov-thm} since $W^p_{X/S}=\tilde\Omega^p_X(\log D)$ whenever $f$ comes from a toroidal pair.
We conjecture the statement of Theorem~\ref{absDegen} to hold also for an arbitrary coherent base $S$ over a field of characteristic zero.
To prove Theorem~\ref{absDegen}, we adapt the proof of the degeneration in \cite{Deligne1987} as follows: since $f$ is proper, 
it suffices to verify 
\begin{equation}
 \sum_{p + q = n} \dime\ R^qf_*W^p_{X/S} = \dime\ R^nf_*W^\bullet_{X/S}. \tag{\textasteriskcentered} \label{ast}
\end{equation}
In \S\ref{TorSpreadOut}, we show that $f: X \to S$ spreads out to a log toroidal family $\phi: \X \to \cS = \Spec (Q \to B)$ where 
$\ZZ \subset B \subset \kk$ is a subring such that $B/\ZZ$ is of finite type. 
Spreading out of log smooth morphisms over a log trivial base has been done before in \cite[Lemma~4.11.1]{Tsuji1999} and we generalize the construction to more general bases and show that the local model in the log toroidal case can be preserved.
Then for suitable fields $k \supset \FF_p$, with $W_2(k)$ denoting the ring of second Witt vectors,
we obtain by base change a diagram with Cartesian squares
\begin{equation}
\tag{SO}\label{SO}
\quad \begin{aligned}
  \xymatrix{
  X \ar[r]\ar_f[d]& \X\ar^{\phi}[d] &\ar[l]\ar^{\phi_W}[d] \X_W & \ar[l] \X_k\ar^{\phi_k}[d] \\
  S \ar[r]& \cS &\ar[l] \Spec W_2(k)& \ar[l] \Spec k. 
  } 
\end{aligned}
\end{equation}

In \S\ref{TorBaCha} we investigate the behavior of $W^\bullet$ under base change which leads to 
equalities like $\dime_\kk R^qf_*W^p_{X/S} = \dime_k R^q(\phi_k)_*W^p_{\X_k/k}$, i.e., it suffices to show \eqref{ast} for 
$\phi_k: \X_k \to \Spec k$. In \S\ref{TorCarIso} we construct the Cartier isomorphism for log toroidal families in 
positive characteristic which we then apply in \S\ref{TorDecompo} to obtain the Frobenius decomposition of 
$F_*W^\bullet_{\X_k/k}$ where $F$ is the relative Frobenius. Finally, in \S\ref{TorAbsDeg}, we put the pieces together and prove Theorem~\ref{absDegen}.

We prove a modest but important generalization of Theorem~\ref{absDegen} to the relative case using Katz's method that first appeared in \cite{Steenbrink1976}.
This requires a detailed understanding of the analytification of the absolute differentials $W^{\bullet,an}_{X}$ with respect to base change as given in \S\ref{subsec-local-analytic} and \S\ref{sec-relative-degen}.
\begin{theorem}\label{rel-degen}
 Let $S = S_m := \Spec (\NN \stackrel{1\mapsto t}\to \CC[t]/(t^{m + 1}))$ and let $f: X \to S$ be a proper log toroidal family with respect to $S \to A_\NN$. Then:
 \begin{enumerate}
  \item $R^qf_*W^p_{X/S}$ is a free $\CC[t]/(t^{m + 1})$-module whose formation commutes with base change.
  \item The spectral sequence $R^qf_*W^p_{X/S} \Rightarrow R^{p + q}f_*W^\bullet_{X/S}$ degenerates at $E_1$.
 \end{enumerate}
\end{theorem}
There are problems with similar theorems in earlier works: the generalization from a one-dimensional base to higher dimensions in \cite[p.\,404]{KawamataNamikawa1994} is flawed which then also affects \cite[Theorem 4.1]{GrossSiebertII}. 
In addition, there is a gap in the proof of \cite[Theorem~4.1]{GrossSiebertII} related to the fact that the de Rham differential of $\Omega^\bullet_{X/S}$ is not $\shO_X$-linear. 
Since our result encompasses the one-parameter base case of \cite[Theorem~4.1]{GrossSiebertII}, Theorem~\ref{rel-degen} closes the latter gap.

\begin{acknowledgement} 
\small The last author feels indebted to Arthur Ogus for almost a decade of communication on the challenges in proving Theorem~\ref{absDegen}.
We thank Mark Gross for connecting the authors with K.\,Chan, C.\,Leung and Z.N.\,Ma whom we also thank for supportive communication and harmonization of our projects. 
We thank H\'el\`ene Esnault and Bernd Siebert for valuable communication and Stefan M\"uller-Stach for bringing the first and last author together. 
Our gratitude for hospitality goes to JGU Mainz and what concerns the last author also to IAS Princeton and Univ.~Hamburg.
\end{acknowledgement}

\begin{conventions} 
We use $\underline X$ to refer to the underlying scheme of a log scheme $X$. 
Given a map $P\ra A$ from a monoid $P$ into the multiplicative monoid of a ring $A$, we refer to the associated log scheme by $\Spec\big(P\ra A\big)$.
\end{conventions}

\section{Generically Log Smooth Families}\label{Pretoroid}
A \emph{log toroidal family} will be a generalization of a saturated log smooth morphism.
We first introduce the weaker notion of a \emph{generically log smooth family} that already enjoys some useful properties.
Log structures in the entire article are assumed to be in the \'etale (or analytic) topology.
If $f: X \to S$ is a finite type morphism of Noetherian schemes, we say a Zariski open $U \subset X$ satisfies the \emph{codimension condition} \eqref{CC} if 
the relative codimension of $Z := X \setminus U$ is $\geq 2$, i.e., for every point $s\in S$ with $X_s,U_s$ the fibers,
\begin{equation}
\tag{CC}\label{CC} \codim{X_s \setminus U_s}{X_s} \geq 2.
\end{equation}
A Cohen--Macaulay morphism is a flat morphism with Cohen--Macaulay fibers.
\begin{definition}
 A \emph{generically log smooth family} consists of: 
 \begin{itemize} 
  \item a finite type Cohen--Macaulay morphism $f: X \to S$ of Noetherian schemes,
  \item a Zariski open $j: U \subset X$ satisfying \eqref{CC},
  \item a saturated and log smooth morphism $f: (U, \M_U) \to (S, \M_S)$ of fine saturated log schemes. 
 \end{itemize}
 The complement $Z := X \setminus U$ we refer to as the \emph{log singular locus} even though $f$ might extend log smoothly to it. 
 We say two generically log smooth families $f,f': X \to S$ with the same underlying morphism of schemes are \emph{equivalent}, if there is some $\tilde U \subset U \cap U'$ satisfying $(CC)$ with $\M_U|_{\tilde U} \cong \M'_{U'}|_{\tilde U}$ 
 compatibly with all data.
\end{definition}
If $T \to S$ is a morphism of fine saturated log schemes, then the base change $X_T \to T$ as a generically log smooth family is defined in the obvious way, taking fiber products in the category of \emph{all} log schemes. 
Note that we need $f: U \to S$ saturated to ensure that $U_T$ is again a fine saturated log scheme. 
The notion of equivalence is due to the fact that we do not care about the precise $U$. However, for technical simplicity we assume some $U$ fixed. 
The name \emph{log singular locus} is in analogy with \cite{GrossSiebertI}. 

\begin{definition}
 For a generically log smooth family $f: X \to S$, the \emph{de Rham complex} is defined as $W^\bullet_{X/S} := j_*\Omega^\bullet_{U/S}$ where $\Omega^\bullet_{U/S}$ denotes the log de Rham complex.
We also define the $\shO_X$-module of degree $m$ log polyvector fields $\Theta^m_{X/S}:=j_*\bigwedge^m\shD er_{U/S}(\shO_U)$.
\end{definition}

\begin{lemma}\label{relaNormal}
Let $f: X \to S$ be a Cohen--Macaulay morphism of Noetherian schemes, and let $j: U \subset X$ satisfy \eqref{CC}. 
Then $j_*\cO_U \cong \cO_X$.
\end{lemma}
\begin{proof}
 This is a special case of \cite[3.5]{Hassett2004}. Note that our \eqref{CC} is a stronger assumption than the condition on the codimension in \cite[3.5]{Hassett2004}.
\end{proof}
Let $X\ra S$ be a generically log smooth family.
Using the language of \cite[Definition~5.9.9]{Grothendieck1965}, a sheaf $\shF$ we call \emph{$Z$-closed} if the natural map $\shF\ra j_*(\shF|_U)$ is an isomorphism. 
Most notably, two $Z$-closed sheaves that agree on $U$ are entirely equal. By their definition, $W^m_{X/S}$ as well as $\Theta^m_{X/S}$ are $Z$-closed.
Furthermore, every reflexive sheaf is $Z$-closed.

\begin{lemma} \label{lem-W-reflexive}
The $\shO_X$-modules $W^m_{X/S}$ and $\Theta^m_{X/S}$ are coherent and reflexive and these depend only on the equivalence class of $f:X\ra S$.
\end{lemma}
\begin{proof}
Let $\tilde U \subset U$ also satisfy $(CC)$.
We have by Lemma~\ref{relaNormal} that $j_*\Omega^\bullet_{\tilde U/S} = j_*\Omega^\bullet_{U/S}$ since $\Omega^m_{U/S}$ is finite locally free. 
Thus $W^\bullet_{X/S}$ depends only on the equivalence class of $f$.
It is clear that it is quasi-coherent.
For every sheaf $\shG$ on $U$, $j_*\shG$ is $Z$-closed, so in particular $W^m_{X/S}$ is $Z$-closed. 
Set $\shF^\vee:=\Hom_{\shO_X}(\shF,\shO_X)$. By Lemma~\ref{relaNormal}, $\shF^\vee$ is $Z$-closed for all $\shF$, so in particular
$(W^m_{X/S})^{\vee\vee}$ is a $Z$-closed sheaf and it coincides with $W^m_{X/S}$ on $U$, hence
$(W^m_{X/S})^{\vee\vee}=W^m_{X/S}$ and $W^m_{X/S}$ is reflexive.
By the extension theorem \cite[9.4.8]{Grothendieck1960a}, there is a coherent $\shG$ that restricts to $W^m_{X/S}$ on $U$. 
Now $\shG^{\vee\vee}=W^m_{X/S}$ since both are $Z$-closed and agree on $U$, hence $W^m_{X/S}$ is also coherent.
The argument for $\Theta^m_{X/S}$ is similar.
\end{proof}

\begin{lemma} \label{lem-W-theta-duals}
$W^m_{X/S}=\shHom(\Theta^m_{X/S},\shO_X)$ and $\Theta^m_{X/S}=\shHom(W^m_{X/S},\shO_X)$.
\end{lemma}
\begin{proof}
The statement is clear on $U$ where all sheaves are locally free and then it follows since all sheaves are $Z$-closed.
\end{proof}

\begin{remark} \label{lem-canonical-log-ext} 
The pushforward $j_*\shM_U\ra j_*\shO_U=\shO_X$ to $X$ yields a log structure which is compatible with $S$, so every generically log smooth family is canonically a log morphism $X\ra S$. We do not know whether this pushforward is compatible with base change (and we do not care). 
\end{remark}

\begin{remark} \label{remark-log-structure-on-X}
In view of Remark~\ref{lem-canonical-log-ext}, neither the so defined log structure $\shM_X$ nor the associated sheaf of log differentials $\Omega_{X/S}$ will be coherent in general, see Example~\ref{ex-standard-example}. 
On the the other hand, $W^m_{X/S}$ and $\Theta^m_{X/S}$ are coherent and have further good properties in the case of log toroidal families as we will see.
\end{remark}
Let $X\ra S$ be a generically log smooth family.
One defines for the log smooth morphism $U\ra S$ the \emph{horizontal divisor} $D_U\subset U$ (see e.g. \cite[Definition~2.4]{Tsu}, also Remark~\ref{rem-horizontal-div} below). This is only a Weil divisor in general. We denote by $D$ its closure in $X$ and by $I_D$ the corresponding ideal sheaf. We define $W^m_{X/S}(-D):=j_*((I_D W^m_{X/S})|_U)$. (This does not need to agree with $I_D W^m_{X/S}$.)

\begin{proposition} \label{prop-top-trivial} 
Let $S=\Spec(\NN\ra \kk)$ for $\kk$ a field where $1\mapsto 0$.
Let $f:X\ra S$ be a generically log smooth family of relative dimension $d$ and let $\omega_f=f^!\shO_S$ denote the (globally normalized) relative dualizing sheaf, then $$W^d_{X/S}(-D)= \omega_f.$$
\end{proposition}
\begin{proof} 
On $U$, this is \cite[Theorem~2.21,\,(ii)]{Tsu} and since both sheaves are $Z$-closed, the statement follows.
\end{proof}

\begin{example}\label{logsmoothsatLisse}
 Let $f: X \to S$ be a log smooth and saturated morphism of Noetherian fine saturated log schemes. 
Then $f$ is flat by \cite[4.5]{kkatoFI} and has Cohen--Macaulay fibers by \cite[II.4.1]{Tsuji1997}. 
We see that $f: X \to S$ gives a generically log smooth family for $U=X$ and $W^\bullet_{X/S}$ is the ordinary log de Rham complex.
\end{example} 
Not every log smooth morphism is saturated, e.g.~see \cite[Remark\,9.1]{Kato1996} for a log smooth morphism that is not even integral. 

\begin{example} \label{ex-Danilov} 
 Let $X/\Spec R$ be a toric variety over a Noetherian base ring $R$.
 The fibers over points in $\Spec R$ are normal (and Cohen--Macaulay), so there is a regular open $U \subset X$ whose 
 complement has relative codimension $\geq 2$ over $\Spec R$. 
For every divisorial log structure on $X$ coming from a torus invariant divisor $D$ on $X$, the map $U\ra \Spec R$ is log smooth and saturated when using the trivial log structure on $\Spec R$. 
Hence $X\ra \Spec R$ is a generically log smooth family. 
The differentials $W^\bullet_{X/S}$ coincide with what is called \emph{reflexive} or \emph{Danilov} or \emph{Zariski--Steenbrink differentials} with log poles in $D$.
This example extends to toroidal pairs $(X,D)$ over $\Spec R$.
\end{example}

\begin{example} 
\label{ex-standard-example}
The $\ZZ[t]$-algebra $A=\ZZ[x,y,t,w]/(xy-tw)$ defines a map $f:\Spec A\ra \AA^1$ that is log smooth and saturated away from the origin when using the divisorial log structure given by $t=0$ 
on source and target, hence a generically log smooth family. The log structure on $\Spec A$ is not coherent at the origin, so $f$ is not log smooth. 
Even worse, $\Omega_f$ is not a coherent sheaf at the origin, see \cite[Example~1.11]{GrossSiebertII}.
\end{example}
Another type of generically log smooth families with application to Gromov--Witten theory can be found in \cite{BN}.

\subsection{Analytification}\label{analytification}

Given a generically log smooth family $f: X \to S$ of finite type over $\CC$, we denote the associated family of complex analytic spaces by $f^{an}: X^{an} \to S^{an}$. 
Induced by $f$, the open $U^{an} \subset X^{an}$ carries an fs log structure so that $U^{an} \to S^{an}$ is a log smooth and saturated morphism of fs log analytic spaces. 
The analog of Lemma~\ref{relaNormal} holds if $X^{an},S^{an}$ are Cohen--Macaulay by \cite[Theorem~3.6]{BanicaStanasila1976}. For $S = \Spec (Q \to A)$ with $A$ an Artinian ring and
$$W^{\bullet,an}_{X/S} := j^{an}_*\Omega^\bullet_{U^{an}/S^{an}},$$
we have $W^{m,an}_{X/S} \cong (W^m_{X/S})^{an}$ since 
both are reflexive coherent $\cO_{X^{an}}$-modules that 
coincide on $U^{an}$. If $f$ is proper then GAGA gives
$H^q(X^{an},W^{p,an}_{X/S}) \cong H^q(X,W^p_{X/S})$ and 
also
$$\HH^{p + q}(X^{an},W^{\bullet,an}_{X/S}) \cong \HH^{p + q}(X,W^\bullet_{X/S}),$$
e.g. via the comparison of the Hodge--de Rham spectral sequences.

\section{Elementary Log Toroidal Families}\label{ElemPretoroid}
For basic notions and constructions of monoids, see \cite{LoAG18}.
\begin{definition} \label{def-ETD}
An \emph{elementary (log) toroidal datum} $(Q\subset P,\F)$ (ETD for short) consists of an injection $Q \hra P$ of sharp toric monoids that turns $P$ into a free $Q$-set whose canonical basis is a union of faces of $P$. We furthermore record a set $\F$ of facets of $P$ with the property that it contains all facets that do not contain $Q$. In other words, if we define
$$\shF_{\min}:= \underbrace{\{F\subset P\hbox{ a facet}\,|\, Q\not\subset F \}}_{\hbox{\tiny vertical facets}},$$
then
$\shF_{\min}\ \subset\ \shF\ \subset\ \shF_{\max}$ where $\shF_{\max}$ is the set of all facets.
\end{definition}
\begin{remark}  \label{rem-horizontal-div}
The facets in $\shF\setminus\shF_{\min}$ will give the \emph{horizontal divisor} that we referred to as $D$ before.
\end{remark}
\begin{figure} 
\includegraphics[width=.9\textwidth]{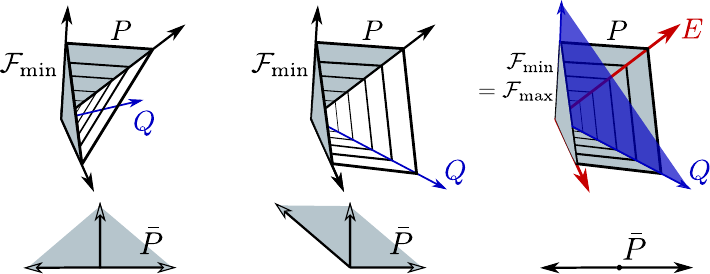}
\caption{Three examples of a saturated injection $Q\subset P$ and the projection $\bar P$, the outer two are log smooth, the middle one gives Example~\ref{ex-standard-example}.}
\label{example-ETDs}
\end{figure}
\begin{lemma}(\cite[Corollary~I.4.6.11, Theorem~I.4.8.14, Corollary~I.1.4.3]{LoAG18}) \label{lemma-eq-to-saturated}
The requirement on the injection $Q\hra P$ in Definition~\ref{def-ETD} is equivalent to saying this map is saturated.
\end{lemma}

See Figure~\ref{example-ETDs} for examples. Even the case $Q=0$ can be interesting since then $\shF_{\min}=\emptyset$.
We denote the union of faces of $P$ that gives the generating set of the free $Q$-action by $E$. 
A face $F$ of $P$ contained in $E$ we call an \emph{essential face}.
Every $p \in P$ has a unique decomposition $p = e + q$ with $e \in E, q \in Q$, hence
\begin{equation} \label{eq-decomp}
E \times Q \to P, \quad (e, q) \mapsto e + q,
\end{equation}
is bijective (\cite[Theorem~I.4.8.14]{LoAG18},\,cf.\,\cite[Lemma~1.1]{Kato2000}). 
Furthermore, we see that $E = P \setminus (Q^+ + P)$ where $Q^+ = Q \setminus 0$ is the maximal ideal. 
Moreover, projecting $E$ to $P^\gp/Q^\gp$ is injective and the set of essential faces gives a fan in $P^\gp/Q^\gp$ whose support $\bar P$ is convex in $(P^\gp/Q^\gp)\otimes_\ZZ\RR$ since it is the convex hull of the projection of $P$. Note that $\bar P^\gp=P^\gp/Q^\gp$.
A choice of splitting $P^\gp\cong\bar P^\gp\oplus Q^\gp$ yields a unique map of sets
$\varphi:\bar P\ra Q^\gp$ so that $\id\times\varphi: \bar P\ra \bar P\oplus Q^\gp$ is a section of the projection $P\ra\bar P$
with the property that its image is $E$, so
\begin{equation}\label{eq-rep-varphi}
P=\{(\bar p,q)\in \bar P\oplus Q^\gp\,|\, \exists \tilde{q}\in Q: q=\varphi(\bar p)+ \tilde{q}\}.
\end{equation}

\begin{lemma} \label{lemma-flat-CM}
The morphism $\underline f:\Spec\ZZ[P]\ra\Spec\ZZ[Q]$ induced by the injection $Q \subset P$ is a Cohen--Macaulay morphism of fiber dimension $d=\rk (P^\gp/Q^\gp)$.
\end{lemma}
\begin{proof} 
Since $P$ is free as a $Q$-set (generated by $E$), $\Spec\ZZ[P]$ is a flat $\Spec\ZZ[Q]$-module.
By \cite[Corollary~6.3.5]{Grothendieck1965} the total space of a faithfully flat morphism of Noetherian schemes is Cohen--Macaulay if and only if the base and all fibers are. By Hoechster's theorem, the fibers of $\Spec\ZZ[P]\ra\Spec\ZZ$ are Cohen--Macaulay, hence $\Spec\ZZ[P]$ and $\Spec\ZZ[Q]$ are Cohen--Macaulay. Now flatness of $\underline f$ implies it is Cohen--Macaulay.
\end{proof}
We next want to define an open set $U$ in the domain of $\underline f$ that satisfies \eqref{CC}.
We will actually define its complement and for this we need a good understanding of the faces of $P$.

\begin{lemma} \label{face-form}
Let $F\subset P$ be a face. Set $\bar F:=F\cap E$, $Q':=Q\cap F$, then
$$F=\bar F+Q' :=\{\bar f+q'|\bar f\in \bar{F},q'\in Q'\}.$$
Since $E$ is a union of faces of $P$, so is $\bar F$. Note also that $Q'$ is a face of $Q$.
\end{lemma}
\begin{proof}
By the decomposition \eqref{eq-decomp}, any element in $F$ has the form $\bar f+q$ with $\bar f \in E, q\in Q$. Since $F$ is a face, $\bar f,q$ are both in $F$, hence $F\subset\bar F+Q'$. The reverse inclusion is clear.
\end{proof}

Consider the set of \emph{bad faces} of $P$ defined as
$$\shB=\left\{ \bar F+Q' \,\left|\, {\bar F\hbox{ is a union of essential faces of rank at most }d-2}\atop{Q'\hbox{ is a face of }Q,\, \bar F+Q'\hbox{ is a face of }P}\right.\right\}.$$ 

Recall that there is a 1-1 correspondence between faces $F$ of $P$ and torus orbits closures $V_F:=\Spec \ZZ[F]$ in $\Spec \ZZ[P]$.
Similarly, for $Q'$ a face of $Q$, we have a torus orbit closure $V_{Q'}:=\Spec \ZZ[Q']\subset\Spec \ZZ[Q]$.

\begin{lemma} \label{lemma-U} 
Given $\bar F+Q'\in\shB$, we find that $V_{\bar F+Q'}$ is flat over $V_{Q'}\subset\Spec\ZZ[Q]$. 
Furthermore, if $X$ is a fiber of $\underline f$, then $\codim{X\cap V_{\bar F+Q'}}{X}\ge 2$.
\end{lemma}

\begin{proof} 
Since $\bar F+Q'$ is free as a $Q'$-set, $\ZZ[\bar F+Q']$ is a free $\ZZ[Q']$-module, so the flatness statement follows. 
The origin $0$ given by the prime ideal $(z^q|q\in Q^+)$ is contained in $V_{Q'}$, let $X_0$ be the fiber over it. 
It suffices to check the codimension condition for this particular fiber.
But note that $X_0\cap V_{\bar F+Q'}= \bigcup_{F\subset \bar F}V_F$ where the union runs over faces $F$ of $P$ contained in $\bar F$ and we have $\dim V_F\le d-2$ by the assumption on $\bar F$.
\end{proof}

Set
\begin{equation} \label{eq-def-U}
U_P := \Spec\ZZ[P]\setminus \left(\bigcup_{B\in\shB} V_B\right).
\end{equation}
For every face $F$ of $P$, we have an open subset $\Spec\ZZ[P_F]$ of $\Spec\ZZ[P]$ where $P_F$ is the localization of $P$ in $F$, i.e.,~$P_F$ is the submonoid of $P^\gp$ generated by $P$ and $-F$.
\begin{lemma} \label{lem-cover-U}
 We find $U_P =\bigcup_{F} U_F$ where the union is over the essential faces $F$ of rank $d-1$.
\end{lemma}
\begin{proof}
Since $U_P$ is a union of torus orbits, it suffices to check that any torus orbit contained in $U_P$ is contained in some $U_F$ for $F$ essential of rank $d-1$. Every torus orbit is given by $O_G:=\Spec\ZZ[G^\gp]$ for $G$ a face of $P$. 
Assume $O_G\subset U$. 
We use Lemma~\ref{face-form} to write $G=\bar G+Q'$. If $\rk\bar G\le d-2$, then $G\in\shB$, so $O_G\not\subset U$. 
Hence, $\dim\bar G\ge d-1$ and $\bar G$ contains some essential face $F$ of rank $d-1$. Then $F$ is also contained in $G$ and thus $O_G$ is contained in $U_F$. Conversely, since $O_F$ is not in any $V_B$, the assertion follows.
\end{proof}

Let $A_Q:=\Spec(Q\ra \ZZ[Q])$ denote the log scheme with standard toric log structure and let 
$A_{P,\shF}$ be the log scheme with underlying scheme $\Spec\ZZ[Q]$ and divisorial log structure given by the divisor $\bigcup_{F\in\shF} \Spec\ZZ[F]$. 
The map $f:A_{P,\shF}\ra A_Q$ induced by $\theta$ is naturally a log morphism by the condition on $\shF$ to contain the vertical faces.
We work here with Zariski log structures which however coincide with the pushforward of the corresponding  \'etale log structures by \cite[Proposition~III.1.6.5]{LoAG18}.
\begin{lemma}[Theorem~3.5 in \cite{kkatoFI} or Theorem~4.1 in \cite{Kato1996}] 
\label{lemma-log-smooth} 
If $\shF=\shF_{\max}$, then $f$ is log smooth.
\end{lemma}

\begin{proposition} \label{ETD-is-pretoroidal}
The map $f:A_{P,\shF}\ra A_Q$ is a generically log smooth family with $U_P$ serving as the specified dense open of log smoothness.
\end{proposition}
\begin{proof} 
If $\shF=\shF_\max$ then $f$ is saturated since $\theta$ is saturated. 
More generally, since $A_{P,\shF_\max}\ra A_{P,\shF}$ is locally given by embedding a face, it is exact. Now by \cite[I.4.8.5(2)]{LoAG18}, $f$ is saturated.

The assertion is clear if $d=0 \iff P=Q$, so assume $d>0$.
Given Lemma~\ref{lemma-flat-CM}, we still need to verify that $U$ satisfies \eqref{CC} and that $f$ is log smooth on $U_P$. 
Note that Lemma~\ref{lemma-U} implies that $U_P$ satisfies \eqref{CC} since the complement of $U_P$ is the union of closed sets each of which has codimension at least two in each fiber.

To see that $f$ is log smooth on $U_P$, by Lemma~\ref{lem-cover-U}, it suffices to check that $f$ is log smooth on $U_F$ for $F$ essential of rank $d-1$. Let $F$ be such a face.
Set $\bar P_F:=P_F/F^\gp$ and note that the projection of $Q$ to $\bar P_F$ is injective because $F^\gp\cap Q=\{0\}$.
There is an isomorphism $P_F\cong F^\gp\times \bar P_F$ commuting with the injection of $Q$ that is $\{0\}\times Q$ on the right.

The log structure on $U_F$ is a divisorial log structure given by a set of divisors each of which pulls back from $\Spec\ZZ[\bar P_F]$, so we may consider the corresponding divisorial log structure on $\Spec\ZZ[\bar P_F]$ to upgrade this to a log scheme $\bar U_F$. 
We have a factorization $U_F\ra \bar U_F\ra A_Q$ with the first map a smooth projection from a product that is therefore strict, hence log smooth. It thus suffices to show that
$\bar U_F\ra A_Q$ is log smooth. Note that $\bar U_F\ra A_Q$ is the log morphism of an ETD with $d=1$. The following lemma finishes the proof.
\end{proof}

\begin{lemma} \label{lemma-d1}
Assume that $f:A_{P,\shF}\ra A_Q$ has one-dimensional fibers (i.e., $d=1$), then $f$ is log smooth. (The third situation of Figure~\ref{example-ETDs} is an example.)
\end{lemma}
\begin{proof}
We are done by Lemma~\ref{lemma-log-smooth} if $\shF=\shF_{\max}$ and this always holds if $Q$ meets the interior of $P$. So assume $Q$ is contained in a proper face of $P$, then by Lemma~\ref{face-form} it is in fact a facet of $P$ and then $\bar P=\NN$ and consequently $P=\NN\times Q$.
A facet of $P$ that is not $Q$ is in 
$\shF_{\min}= \{\NN\times F\,|\, F\hbox{ is a facet of Q}\}$. 
Hence $\shF\subsetneq\shF_{\max}$ implies $\shF=\shF_{\min}$ and thus $f$ is strict. 
Since $\underline f$ is smooth, we find $f$ is log smooth.
\end{proof}

\begin{corollary}\label{smooth-type-decompo}
It is possible to find open subsets $U_1$ and $U_2$ so that $U_P = U_1 \cup U_2$ and $A_P|_{U_1} = A_{P,\F}|_{U_1}$ and 
 $f: U_2 \subset A_{P,\F} \to A_Q$ is strict and smooth.
\end{corollary}
\begin{proof}
 Let $\E_1$ be the set of essential faces of rank $d - 1$ such that when applying the proof of Lemma 
 \ref{lemma-d1} to $\bar U_F \to A_Q$ from the proof of the proposition, we are in the case 
 $\F = \F_{max}$, and let $\E_2$ be the set of faces where we are in case $\F = \F_{min}$.
 Then for $F \in \E_1$ we have $A_P|_{U_F} = A_{P,\F}|_{U_F}$, and for $F \in \E_2$, the morphism 
 $U_F \to A_Q$ is strict and smooth. Now we define $U_1 = \bigcup_{F \in \E_1} U_F$ and 
 $U_2 = \bigcup_{F \in \E_2} U_F$.
\end{proof}

\begin{example}\label{ex etd}
If $(Q\subset P,\F)$ is an ETD and $r\ge 0$, then we obtain another ETD $(Q\times\{0\}\subset P\times\NN^r,\shF')$ where $\shF'=\{F\times \NN^r\,|\,F\in\shF\}$.
\end{example}

\section{Log Toroidal Families}\label{Toroidal}
We define log toroidal families and investigate their basic properties.

\begin{definition}\label{toroidDef}
We say that a generically log smooth family $f: X \to S$ is \emph{log toroidal} if for every geometric point $\bar x \to X$, we have a commutative diagram 
\begin{equation}
\tag{LM}\label{LM}
\quad \begin{aligned}
  \xymatrix@R-2pc{
   & (V,g^{-1}(U)) \ar[ldd]_g \ar@{.>}[ddr]^h \ar[dddd] & & \\
   \\
   (X,U) \ar[dddd]_f & & (L, U_L) \ar[ddl] \ar[ddr]^c & \\ 
   \\
   & \tilde S \ar[ldd] \ar[ddr]^a & & (A_{P,\F}, U_P) \ar[ddl]  \\
   \\
   S & & A_Q & \\
  } 
\end{aligned}
\end{equation}
where $g: V \to X$ is an \'etale neighborhood of $\bar x$, $\tilde S \to S$ is a strict \'etale neighborhood of $f(\bar x)$ and $a$ is given by a chart 
$Q\ra \M_{\tilde S}$ of $\tilde S$. The bottom right diagonal map is required to be given by an ETD $(Q\subset P,\F)$ and $U_{P}\subset A_{P,\F}$ denotes the open set from \eqref{eq-def-U}.
The solid arrows are morphisms of schemes and log morphisms on the specified opens, whereas $h: V \to L$
 is an \'etale morphism only of underlying schemes. The bottom right diamond is Cartesian, in particular $U_L=c^{-1}(U_{P})$. Moreover, we have an open $\tilde U \subset V$ 
 satisfying \eqref{CC}, such that $\tilde U \subset g^{-1}(U) \cap h^{-1}(U_L)$ and there is an isomorphism $g^*\M_X \cong h^*\M_L$ of the two log structures on $\tilde U$ compatible with the maps to $S$. 

The diagram \eqref{LM} is called a \emph{local model} for $f: X \to S$ at $\bar x$.
If $S \cong \Spec (Q \to B)$, every point has a local model with $\tilde S=S$ and $a$ is given by the chart $Q \to B$ then we say $f: X \to S$ is \emph{log toroidal with respect to $a: S \to A_Q$}.
 \end{definition}

Log toroidal families are stable under strict base change.
\begin{remark}\label{local-model-localize} 
Note that Definition~\ref{toroidDef} only requires a covering of $X$ by \eqref{LM} but does not say that an arbitrary geometric point $\bar x\in X$ permits a diagram \eqref{LM} that identifies $\bar x$ with the origin in $A_P$.
However, if $\kk$ is algebraically closed, one can show that by localizing the ETD in \eqref{LM} and using Example~\ref{ex etd} one can assume that $\bar x\in X$ becomes the origin in $A_P$. We will make use of this fact in the proof of Theorem~\ref{rel-degen}.
\end{remark}

\begin{example} \label{ex-ETD} 
Every elementary log toroidal family $f: A_{P,\F} \to A_Q$ is log toroidal.
\end{example}

\begin{example} \label{example-toroidal}
The generically log smooth families given in Example~\ref{ex-Danilov} are log toroidal families with $Q=0$ in every ETD.
\end{example}

\begin{example}
 A saturated log smooth morphism $f: X \to S$ is log toroidal with $U = X$. Indeed, locally starting from a neat chart of $f$, set $\shF=\shF_\max$ and then apply Example~\ref{ex etd} to have local models.
That this works is \emph{not} a trivial consequence of Theorem 3.5 in \cite{kkatoFI}. Instead, use \cite[Theorem~VI.3.3.3]{LoAG18}.
\end{example}

\begin{example}\label{example-GS}
 In the setting and notation of the Gross--Siebert program, \cite[Theorem~2.6]{GrossSiebertII} shows that if $(B,\cP)$ is positive and simple, and $s$ is lifted open gluing data, then $X_0^\dagger(B,\cP,s) \to \Spec( \NN\ra k)$ is log toroidal. 
More generally, it was shown in \cite[Proposition~2.8]{rud} that c.i.t.~log Calabi--Yau spaces are log toroidal over $\Spec (\NN\ra k)$.
The divisorial deformations defined in \cite{GrossSiebertII} are also log toroidal families.
\end{example}

\section{Log Structures and Infinitesimal Deformations} 
\label{sec-log-infini}
Let $X$ be a toroidal crossing space over a field $\kk$. 
As mentioned in the introduction, $X$ can be equipped with a sheaf of sets $\shL\shS_X$ which we recall next.
An alternative categorical approach to study log smooth morphisms for a fixed underlying morphism has been developed in \cite{Olsson2003} though we follow the sheaf approach.
Let $S=\Spec(\NN\stackrel{1\mapsto 0}\ra \kk)$ be the standard log point.
The pair $(\shP,\one)$ gives a \emph{ghost structure} in the sense of \cite[Definition~3.16]{GrossSiebertI}. 
Indeed, the \emph{type} of the ghost structure is fixed by requiring it to be the one given by the local chart that comes with the definition of a toroidal crossing space. We will refer to this type as the \emph{given type} below.
By \cite[Definition~3.19 and Proposition~3.20]{GrossSiebertI}, there is a sheaf $\shL\shS_{X}$ (denoted $\shL\shS_{X^g}$ in loc.cit.)~with the property that for every \'etale open $U\subset X$, there is a natural bijection
$$
\resizebox{\textwidth}{!}{
$
\Gamma(U,\shL\shS_X) = \left\{
\begin{array}{c}
\shM_U\ra\shO_U\hbox{ a log structure of}\\
\hbox{the given type},\tilde\one\in\Gamma(U,\shM_U),\\ \ol\shM_U\stackrel{\sim}\ra\shP\hbox{ an isomorphism}
\end{array}
\left|
\begin{array}{c}
{(U,\shM_U)\ra S \hbox{ via } 1\mapsto\tilde\one\hbox{ is a}}\\
{\hbox{log smooth morphism and}}\\
{\shM_U\ra\ol\shM_U\ra\shP\hbox{ sends }\tilde\one\hbox{ to }\one}
\end{array}
\right.\right\}
$
}
$$
where the set on the right is to be taken modulo isomorphisms. 
The support of $\shP/\one$ agrees with $X_\sing$, so the sheaf $\shL\shS_{X}$ is supported on $X_\sing$. 

Set $S_\eps:=\Spec(\NN\stackrel{1\mapsto \eps}\ra \kk[\eps]/(\eps^2))$. 
If $V\ra S$ is a log smooth morphism with $V$ affine, then there is a unique log smooth lifting $V_\eps\ra S_\eps$ up to isomorphism.
For $(\M,\tilde\one) \in \cL\cS_X(U)$ and an affine $V \subset U$, the deformation $i:V\ra V_\eps$ yields an extension
\begin{equation} \label{log-smooth-extension-sequence}
0 \to \cO_V \to i^*\Omega^1_{\underline {V_\eps}} \to \Omega^1_{\underline V} \to 0
\end{equation}
where on the left $1 \mapsto i^*d\eps$. 
The classes of such local extensions glue to a well-defined class in 
$\E xt^1(\Omega^1_U,\cO_U)$ (though neither the extensions nor the deformations need to glue).
We have thus defined a map of sheaves of sets 
\begin{equation}\label{define-eta}
\eta: \cL\cS_X \to \E xt^1(\Omega^1_X, \cO_X)=\shT^1_X.
\end{equation}

\begin{remark}
In this form, the map $\eta$ seems to be new. However, a close relationship between log structures and $\shT^1_X$ has been observed before in \cite[Proposition\,1.1]{KawamataNamikawa1994}, \cite[Remark~(3.11)]{Steenbrink1995}, \cite[Theorem\,11.7]{Kato1996}, \cite[Example~3.30]{GrossSiebertI}, \cite[Theorem~3.18]{Olsson2003}, \cite[Theorem~7.5]{sch-sie}. 
\end{remark}

Both sheaves in \eqref{define-eta} have a natural action of $\cO^\times_X$: indeed, $\shT^1_X$ because it is coherent and $\cL\cS_X$ for we let a section $\lambda$ of $\cO^\times_X$ act by $\tilde\one\mapsto \lambda^{-1}\tilde\one$.
\begin{proposition} \label{prop-eta-equiv}
The map $\eta$ is $\cO^\times_X$-equivariant.
\end{proposition}
\begin{proof} 
At a geometric point $\bar x\in X$ with $M=(\shM,\tilde\one)\in\shL\shS_{X,\bar x}$ for $\shM$ defined on some \'etale $U\ra X$ that contains $\bar x$, let $\mu_M:\shO_{X,\bar x}\ra \shT^1_{X,\bar x}$ denote the connecting
\\[1mm]
    \begin{minipage}{.56\linewidth}
 homomorphism at $\bar x$ in the long exact sequence obtained from applying $\cH om(-,\cO_X)$ to \eqref{log-smooth-extension-sequence}. 
By a general fact for extensions, we have $\mu_M(1)=\eta(M)$. For $\lambda\in\shO_{X,\bar x}^\times$, let $M_\lambda\in\shL\shS_{X,\bar x}$ denote the element $(\shM,\lambda^{-1}\tilde\one)$. 
The statement of the lemma comes down to the following claim.
    \end{minipage}
    \begin{minipage}{.44\linewidth}
      \centering
$$
 \xymatrix@R-1pc@C-.7pc{
 &U_1\ar@{.>}[ddl] \ar^{i_1}[r] \ar|\hole[ddr]
& (U_1)_\eps\ar[ddr]\\
 &U_\lambda
 \ar@{}|(.46){\phantom{xi}\chi}[u]^(.1){}="a"^(.95){}="b" \ar@{.>} "a";"b"
\ar@{.>}[dl] \ar^(.55){i_\lambda}[r]\ar[dr]& (U_\lambda)_\eps\ar[dr]\\
\underline S&&\ar[ll]S\ar[r]& S_\eps \ar@{.>}@/^{1pc}/[lll]
 }
$$
\end{minipage}
 \begin{claim} 
 $\mu_M(\lambda)=\eta(M_\lambda)$. 
 \end{claim}
To prove the claim, let $U_1,U_\lambda$ denote the log smooth schemes over $S$ respectively obtained from the log scheme $U$ and the map to $S$ given by $1\mapsto \one$ and $1\mapsto \lambda^{-1}\one$ respectively.
Let $(U_1)_\eps$ and $(U_\lambda)_\eps$ be the unique deformations of $U_1,U_\lambda$ over $S_\eps$ respectively. 
Let $\chi:U_\lambda\ra U_1$ be the canonical isomorphism over $\underline S=\Spec\big(0\ra \kk)$.

We are now going to use facts about idealized log schemes, see \cite[III.1.3 \& Variant 3.1.21]{LoAG18} for an introduction.
We give $S_\eps$ the ideal $\langle 2\rangle$ generated by $2\in\NN$ and $(U_1)_\eps$ and $(U_\lambda)_\eps$ the pullback ideals $K_1,K_\lambda$ respectively so that  $((U_1)_\eps,K_1)$ and $((U_\lambda)_\eps,K_\lambda)$ are ideally log smooth over $(S_\eps,\langle 2\rangle)$.
The map $(S_\eps,\langle 2\rangle)\ra (A_\NN,\emptyset)$ is an \'etale map of idealized log schemes and $A_\NN\ra\underline S$ is log smooth, hence the composition $\pi:((U_1)_\eps,K_1)\ra(S_\eps,\langle 2\rangle)\ra\underline S$ is log smooth. 
We apply the infinitesimal lifting property to the diagram
\[
 \xymatrix@R-1pc@C-.7pc{
(U,K)\ar^{i_1}[r]\ar^{i_\lambda}[d]& ((U_1)_\eps,K_1)\ar^{\pi}[d]\\
((U_\lambda)_\eps,K_\lambda) \ar[r]& \underline S
 }
\]
where $(U,K)$ is the idealized log scheme $U=U_1\stackrel{\chi}{=}U_\lambda$ with ideal given by $(\tilde\one)^2$ or equivalently $(\lambda^{-1}\tilde\one)^2$. 
The left vertical map $i_\lambda$ is strict for the log structure and ideal and given by a square-zero ideal.
We obtain a morphism $\tilde\chi: (U_\lambda)_\eps\ra (U_1)_\eps$ of log schemes that preserves the ideals and is an isomorphism on ghost sheaves.
Consequently, with $\rho_\lambda\in \shM_{(U_\lambda)_\eps,\bar x}$ and $\rho_1\in \shM_{(U_1)_\eps,\bar x}$ the images of the generator $1\in\shM_{S_\eps}$ respectively, we have
$\tilde\chi^*\rho_1=\tilde\lambda\cdot\rho_\lambda$ for some $\tilde\lambda\in\shO_{(U_\lambda)^\times_\eps,\bar x}$ that restricts to $\lambda\in\shO^\times_{U,\bar x}$.
This implies that $\tilde\chi$ becomes an isomorphism after shrinking $(U_\lambda)_\eps$, $(U_1)_\eps$ if needed. 
Using $i_1 \circ \chi = \tilde \chi \circ i_\lambda$, we obtain the commutative diagram
\[
 \begin{CD}
 \eta(M_\lambda): \qquad @. 0 @>>> \cO_U @>1 \mapsto i_\lambda^*d\alpha(\rho_\lambda)>> i_\lambda^*\Omega^1_{\underline{(U_\lambda)_\eps}} @>>> \Omega^1_{\underline U} @>>> 0 \\
 @. @. @VV\lambda^{-1} \cdot V @| @| @. \\
 \eta(M_1): \qquad @. 0 @>>> \cO_U @>1 \mapsto i_\lambda^*d(\tilde \lambda\alpha(\rho_\lambda))>> i_\lambda^*\Omega^1_{\underline{(U_\lambda)_\eps}} @>>> \Omega^1_{\underline U} @>>> 0, \\
 \end{CD}
\]
and we conclude $\eta(M_\lambda) = \mu_M(\lambda)$ via standard homological algebra.
\end{proof} 

\begin{lemma} \label{lemma-injective}
Let $\bar x\in X$ be a geometric point with $\kk[\shP_{\bar x}]$ smooth, then 
\begin{enumerate}
\item for $M\in\shL\shS_{X,\bar x}$, the map $\mu_{M,\bar x}:\shO_{X,\bar x}\ra\shT_{X,\bar x}$ is surjective,
\item $\shO^\times_{X,\bar x}$ acts transitively on $\shL\shS_{X,\bar x}$,
\item $\eta_{\bar x}:\shL\shS_{X,\bar x} \ra \shT_{X,\bar x}$ is injective.
\end{enumerate}
\end{lemma}
\begin{proof} 
Set $P:=\shP_{\bar x}$. For (1), let $U\ra X$ be an \'etale affine neighborhood of $\bar x$ where $M=(\shM_U,\one_U)$ is defined and $h:(U,\shM_U)\ra\Spec\big(P\ra\kk[P]/(z^{\one_{\bar x}})\big)$ the strict $S$-morphism whose underlying map is smooth.
Possibly after shrinking $U$, via $\eps\mapsto \one_{\bar x}$, we obtain a strict map of extensions over $S_\eps$,
$$h_\eps:(U_\eps,\shM_{U_\eps})\ra\Spec\big(P\ra\kk[P]/(z^{(\one_{\bar x}+\one_{\bar x})})\big)$$
whose underlying morphism is also smooth and hence $\Omega_{\underline{U_\eps}}$ is locally free. 
This implies that the corresponding term $\E xt^1(\Omega_{\underline{U_\eps}},\shO_X)_{\bar x}$ in the long exact sequence for \eqref{log-smooth-extension-sequence} vanishes and thus $\mu_{M,\bar x}$ is surjective.

To show (2), note that it suffices to show that any two elements in $\shL\shS_{X,\bar x}$ are isomorphic over $\underline{S}$. 
Equivalently by \cite[Definition~3.19 \& Corollary~3.12]{GrossSiebertI}, the composition 
$$\shL\shS_{X,\bar x}\subset \shE xt^1(\shP^\gp_{\bar x}/\ZZ\one_{\bar x},\shO^\times_{X,\bar x})\ra \shE xt^1(\shP^\gp_{\bar x},\shO^\times_{X,\bar x})$$ 
needs to be the constant map. 
By assumption, $P$ is free and then (2) follows from the description of $\shE xt^1(\shP^\gp_{\bar x},\shO^\times_{X,\bar x})$ in \cite[Proposition~3.14]{GrossSiebertI}.

For (3), if $\bar x\not\in X_\sing$ both stalks are trivial and there is noting to show, so assume $\bar x\in X_\sing$. 
By \cite[Proposition~1.10]{fri}, we have $\shT^1_{X,\bar x}\cong \shO_{X_\sing,\bar x}$, so the kernel of the action of $\shO^\times_{X,\bar x}$ on $\shT^1_{X,\bar x}$ is $K:=\ker\big(\shO^\times_{X,\bar x}\ra \shO^\times_{X_\sing,\bar x}\big)$. 
If we show that $K$ is contained in the kernel of the action of $\shO^\times_{X,\bar x}$ on $\shL\shS_{X,\bar x}$, then (3) follows from (2) and Proposition~\ref{prop-eta-equiv}. 
By assumption, $X$ is normal crossings at $\bar x$. Let $X_1,...,X_r$ be the local components of $X$ at $\bar x$, $r\ge 2$. 
Let $\lambda\in K$ be given and write $\lambda=1+\sum_{i=1}^r f_i$ where $f_i|_{X_j}=0$ for $i\neq j$.
We observe that $\lambda=\prod_i (1+f_i)$ because $f_if_j=0$ for $i\neq j$. 
If $\NN^r\ra\shO_{X,\bar x},\,e_i\mapsto h_i$ is a chart of $X$ at $\bar x$ representing an element of $\shL\shS_{X,\bar x}$ with $\one=\sum_i e_i$ and $V(h_i)=X_i$, then
$e_i\mapsto (1+f_i) e_i$ defines an automorphism of $\shM_{X,\bar x}$ compatible with the map to $\shO_{X,\bar x}$ because $(1+f_i)h_i=h_i$. 
It takes $\one$ to $\lambda\one$, so $\lambda$ acts trivially on $\shL\shS_{X,\bar x}$.
\end{proof}

\begin{remark} \label{remark-eta-may-be-zero}
For $\kappa\ge 2$, consider the monoid $P_\kappa=\langle e_1,e_2,\one| e_1+e_2=\kappa\one \rangle$ and the toroidal crossing space $X=\Spec\big(P_\kappa\ra\kk[P_\kappa]/(z^\one)\big)$. 
The map $\eta:\shL\shS_X\ra\shT^1_X$ is the zero map $\kk^\times\ra\kk$, so the smoothness assumption in Lemma~\ref{lemma-injective} is necessary.
\end{remark}

\begin{theorem} \label{thm-LS-T1}
Let $X$ be a toroidal crossing space with $P_{\bar x}\cong \NN^2$ whenever $\bar x$ is the generic point of a component of $X_\sing$ then $\eta:\shL\shS_X\ra\shT^1_X$ is injective.
On the open set $V\subset X$ of points $\bar x$ with $\shP_{\bar x}\cong \NN^r$ for some $r$, we have $\eta(\shL\shS_V)=(\shT_V^1)^\times$ where
$(\shT_X^1)^\times\subset \shT_X^1$ denotes the subsheaf of those elements that generate $\shT_X^1$ as an $\shO_X$-module.
\end{theorem}
\begin{proof} 
The second statement is Lemma~\ref{lemma-injective}. 
For the first statement also follows from the Lemma combined with the fact that for every open $U\subset X$, the restriction map\\[1mm]
\begin{minipage}{0.56\textwidth}
$\shL\shS_X(U)\ra \shL\shS_X(U\cap V)$ is injective which is a consequence of Corollary~\ref{injection-into-stalks-at-components} below. 
Indeed, in view of the diagram on the right, that the composition of the left vertical and bottom horizontal arrow is injective implies the injectivity
\end{minipage}
\begin{minipage}{.44\linewidth}
$$
\xymatrix@R-0pc@C-.7pc{
\shL\shS_X(U)\ar[d]\ar^{\eta}[r]&  \shT^1_X(U)\ar[d]\\
\shL\shS_X(U\cap V)\ar^{\eta}[r]& \shT^1_X(U\cap V).
}$$
\ 
\end{minipage}\\[.7mm]
of the top horizontal arrow.
\end{proof}


\section{Toroidal Crossing Spaces as Log Toroidal Families} \label{sec-log-TC}
Let $X$ be a toroidal crossing space.
Let $\bar x$ be geometric point and $V_{\bar x}$ the \'etale neighborhood with a smooth map $V_{\bar x}\ra\Spec \kk[\shP_{\bar x}]/z^\one$ that exists by the definition of $X$.
Set $N=\shP^\gp_{\bar x}$ and $M_\RR=\Hom(N,\RR)$. 
We obtain a lattice polytope
$\sigma_{\bar x}=\{m\in M_\RR\,|\, m|_{\shP_{\bar x}}\ge 0,\one(m)=1\}$ (we use that $X$ is reduced here).
For a face $\tau\subset\sigma_{\bar x}$, we denote by $V_\tau$ the inverse image of the closed subset $\Spec \kk[\tau^\perp\cap \shP_{\bar x}]$ of $\Spec \kk[\shP_{\bar x}]/z^\one$ in $V_{\bar x}$.
Theorem 3.22 in \cite{GrossSiebertI} says the following.
\begin{theorem}[Gross--Siebert] \label{thm-grosie-LS}
$\shLS_X|_{V_{\bar x}}$ is isomorphic to a subsheaf of
$\oplus_{\omega}\cO^\times_{V_\omega}$ where the sum is over the edges of $\sigma_{\bar x}$.
The sections of the subsheaf on an open $V\subset V_{\bar x}$ are given as the tuples $(f_\omega)_\omega$ so that, for every two-face $\tau$ of $\sigma_{\bar x}$, we have
\begin{equation}\label{eq log}
\prod_{\omega\subset\tau}d_\omega \otimes f_\omega^{\epsilon_\tau(\omega)}|_{V_\tau}=1
\end{equation}
as an equality in $M\otimes_\ZZ\Gamma(V,\cO^\times_{V_\tau})$ where $d_\omega$ is a primitive generator of the tangent space to $\omega$ and 
$\epsilon_\tau(\omega)\in\{-1,+1\}$ is such that $(\epsilon_\tau(\omega)d_\omega)_{\omega\subset\tau}$ gives an oriented boundary of $\tau$.
\end{theorem}
\begin{corollary} \label{injection-into-stalks-at-components}
Given an \'etale open $U\ra X$, the natural map $\shL\shS_X(U)\ra\prod_{\bar x}\shL\shS_{X,\bar x}$, for the product running over the generic points $\bar x$ of the components of $U_\sing$, is injective.
\end{corollary} 
The isomorphism in the theorem naturally depends on the morphism $V_{\bar x}\ra\Spec \kk[\shP_{\bar x}]/z^\one$ in a way that enables the following result.
\begin{corollary} 
For each irreducible component $X_{\omega}$ of $X_\sing$ there is an $\shO_{\tilde X_{\omega}}^\times$-torsor $\shN_{{\omega}}^\times$ on its normalization $\tilde X_{\omega}$ so that 
$$\shLS_X\subset \bigoplus_{X_{\omega}} q_{\omega,*}\shN_{{\omega}}^\times$$
for $q_{\omega}:\tilde X_{\omega}\ra X_{\omega}$ the normalization, and the subsheaf is locally characterized by
Theorem~\ref{thm-grosie-LS} when using suitable local trivializations of the torsors.
\end{corollary}
Let $\shN_{{\omega}}$ denote the associated line bundle so that $\shN_{{\omega}}^\times$ is its $\shO_{\tilde X_{\omega}}^\times$-torsor of generating sections. 
We therefore obtain an injection of $\shLS_X$ in the coherent sheaf $\bigoplus_{X_{\omega}} q_{\omega,*}\shN_{{\omega}}$.
\begin{lemma} \label{lemma-LS-equivariant}
Under the hypothesis of Theorem~\ref{thm-LS-T1}, the injection $\shL\shS_X\hra\bigoplus_{X_{\omega}} q_{\omega,*}\shN_{{\omega}}$ is $\shO^\times_X$-equivariant.
\end{lemma}
\begin{proof} 
We borrow the notation $P_{\kappa}$ from Remark~\ref{remark-eta-may-be-zero}.
From a careful analysis of the proof of \cite[Theorem~3.22]{GrossSiebertI} one finds that the action $\one\mapsto \lambda^{-1}\one$ becomes $f_\omega\mapsto \lambda^{\kappa_\omega}f_\omega$ where $\kappa_\omega$ is such that $\shP_{\bar x}\cong P_{\kappa_\omega}$ at the generic point $\bar x$ of $X_\omega$.
Indeed, if a local model at $\bar x$ is given by $xy=f_\omega (z^\one)^{\kappa_\omega}$, this is equivalent to $xy=\lambda^{\kappa_\omega} f_\omega (\lambda^{-1}z^\one)^{\kappa_\omega}$ which explains the action.
The hypothesis of Theorem~\ref{thm-LS-T1} says that $\kappa_\omega=1$ for all $\omega$, so indeed the action of $\shO^\times_X$ on $\shL\shS_X$ is compatible with the ordinary action on the coherent sheaf $\bigoplus_{X_{\omega}} q_{\omega,*}\shN_{{\omega}}$.
\end{proof}
\begin{theorem} \label{thm-inject-T1}
If $X$ is a normal crossing space, then the injection in Lemma~\ref{lemma-LS-equivariant} factors as the composition of $\eta:\shL\shS_X\ra\shT^1_X$ and a uniquely determined injection of coherent sheaves
$\shT^1_X\hra \bigoplus_{X_{\omega}} q_{\omega,*}\shN_{{\omega}}$.
\end{theorem} 
\begin{proof} 
Given Lemma~\ref{lemma-LS-equivariant} and Theorem~\ref{thm-LS-T1} and noting that $V=X$ for a normal crossing space and that the annihilator of $\shT^1_X$ is contained in the annihilator of $\bigoplus_{X_{\omega}} q_{\omega,*}\shN_{{\omega}}$, the statement becomes an elementary lemma about a cyclic module whose proof we omit.
\end{proof}
\begin{definition} 
For a point $\bar x\in X$, let $X^\circ_{\bar x}\subset X$ denote the Zariski locally closed subset where $\shP$ is locally constant with stalk $\shP_{\bar x}$, so that $X$ is the disjoint union of $X^\circ_{\bar y}$ for suitable points ${\bar y}$. We call the closure $X_{\bar x}$ of $X^\circ_{\bar x}$ the \emph{stratum} of ${\bar x}$ which again decomposes into $X^\circ_{\bar y}$. We infer the notion of strata to the normalization of $X$.

A section of $s\in\Gamma(U,\shL\shS_X)$ for a Zariski open $U\subset X$ is called \emph{sch\"on} if it extends to a section $(s_\omega)_\omega\in \Gamma(X,\bigoplus_{X_{\omega}} q_{\omega,*}\shN_{{\omega}})$ so that, for each $\omega$, the vanishing locus $\tilde Z_\omega$ of $s_\omega$ in $\tilde X_\omega$ is reduced, does not contain any strata and has regular intersection with each stratum inside $\tilde X_\omega$ (in particular $\tilde Z_\omega\cap X_{\omega}^\circ$ is smooth).
We also assume that $Z=\bigcup_\omega q_{\omega}(\tilde Z_\omega)$ is the complement of $U$ (otherwise $U$ can be enlarged).
\end{definition} 

\begin{definition} \label{def-simple}
A sch\"on section is called \emph{simple} if for every closed point $\bar x\in X$ with $V_{\bar x}\ra\Spec \kk[\shP_{\bar x}]/z^\one$ the smooth map from a neighborhood, we have the following situation. 
Let $Z\cap V_{\bar x}=\bigcup_{\omega\in \Omega}  Z_\omega$ be the local decomposition of $Z$ into irreducible components where we may assume each $Z_\omega$ contains $\bar x$.
\begin{enumerate}
\item There is a disjoint union $\Omega=\Omega_1\sqcup...\sqcup \Omega_q$ with $q<\rk \shP_{\bar x}$ such that $Z_i:=Z_\omega\cap X_{\bar x}=Z_{\omega'}\cap X_{\bar x}$ whenever $\omega,\omega'$ are in the same $\Omega_i$.
\item $Z_1,...,Z_q$ form a collection of normal crossing divisors in $X_{\bar x}$ at ${\bar x}$.
\item for each $i$, the primitive vectors $d_\omega$ for $\omega\in \Omega_i$ are the set of edge vectors of an elementary simplex $\Delta_i\subset N_\RR$. (A lattice simplex is \emph{elementary} if its vertices are the only lattice points contained in it.)
\end{enumerate}
\end{definition}
We remark that if $q_\omega:\tilde X_\omega\ra X_\omega$ is not an embedding, the zero set $\tilde Z_\omega$ of $s_\omega$ may locally contribute two or more components of $Z$ at a point $\bar x$ which may or may not lie in different $\Omega_i$.
\begin{theorem}[Gross--Siebert] 
A toroidal crossing space $X$ over an algebraically closed field $\kk$ together with simple~section $s\in\Gamma(U,\shL\shS_X)$ gives $X$ the structure of a log toroidal family over $S=\Spec\big(\NN\ra\kk\big)$ with $U$ the locus of log smoothness.
\end{theorem}
\begin{proof}
Using assumptions in Definition~\ref{def-simple}, the proof is the same as the one of \cite[Theorem~2.6]{GrossSiebertII}. See also Example \ref{example-GS}.
\end{proof}
We remark that the $\Delta_i$ give the local structure of the singularities in the nearby fiber, cf.~\cite[Proposition~2.2]{GrossSiebertII}.
We also remark that all ETDs have $\shF=\shF_{\min}$, i.e.,~there is no horizontal divisor. Proposition~\ref{prop-top-trivial} implies $W^{\dim X}_{X/S}=\omega_{X/S}$.

\begin{proposition} \label{nc-to-log-toroidal}
A normal crossing space $X$ with $X_\sing$ projective and $\shT_X^1$ generated by global sections permits a dense open $U$ and a simple section $s\in\Gamma(U,\shL\shS_X)$.
In view of \ref{def-simple}, we have $q=1$ at every point in $Z$ and $\Delta_1$ in each ETD is a standard simplex which means all ETDs have smooth nearby fibers.
\end{proposition}
\begin{proof} 
Applying Bertini's theorem to the line bundle $\shT_X^1$ on $X_\sing$, we obtain a section $\hat s\in\Gamma(X_\sing,\shT_X^1)$ that gives a simple section $s\in\Gamma(X\setminus V(\hat s),\shL\shS_X)$ by Theorem~\ref{thm-inject-T1}. 
\end{proof}

\begin{proposition} \label{prop-tc-implies-nc}
Theorem~\ref{maintheorem-nc} follows from Theorem~\ref{maintheorem-tc}.
\end{proposition} 
\begin{proof} 
We are given $E$ that is transverse to the strata of $X$. 
We apply a slight variant of Proposition~\ref{nc-to-log-toroidal} by making sure the zero locus $Z$ of the section $\hat s$ generated by Bertini is transverse also to $E$.
Theorem~\ref{maintheorem-tc} gives an orbifold smoothing but we know it is an actual smoothing from the fact that each $\Delta_1$ is standard.
\end{proof}

The next two lemmata reduce Theorem~\ref{maintheorem-tc} to the log Calabi--Yau case, i.e., to the case $W^d_{X/S} \cong \cO_X$. We achieve this by modifying the log structure so that the new family is log Calabi--Yau.

\begin{lemma} \label{lemma-add-E}
Let $f:X\ra S$ be a log toroidal family with empty horizontal divisor. 
Let $E\subset \underline X$ be a Cartier divisor that meets all strata and $Z$ transversely, 
i.e., locally along $E$ the triple $(X,Z,E)$ is \'etale equivalent to $(E\times \AA^1,(E\cap Z)\times\AA^1,E\times\{0\})$.
There is a new log toroidal family $X(\log E)\ra S$ that has $E$ as its horizontal divisor and factors through $f$ (by forgetting $E$), so in particular $W^{\dim X}_{X(\log E)/S}(-E)=\omega_{X/S}$.
\end{lemma}
\begin{proof} 
On $U$ the result is straightforward and along $Z$ we use the product description to make $E$ the horizontal divisor in the ETDs by adding a summand $\NN$ to $P$ and the unique new facet gets included in $\shF$. That these give local models follows the same proof as \cite[Theorem~2.6]{GrossSiebertII} noting that we may treat the local equation for $E$ as one of the $f_i$ in the notation of loc.cit..
\end{proof}
\begin{lemma} \label{W-dim-trivial}
Let $f:X\ra S$ be a projective log toroidal family with empty horizontal divisor and assume that $\omega^{-1}_{X/S}$ is generated by global sections, then $\omega^{-1}_{X/S}\cong \shO_X(E)$ for a divisor $E$ that satisfies the assumption of Lemma~\ref{lemma-add-E}. 
In particular, $W^{\dim X}_{X(\log E)/S}\cong \shO_X$.
\end{lemma} 
\begin{proof} 
This follows via an application of Bertini's theorem.
\end{proof}

In general we do not know if deformations of log toroidal families are locally unique. The following theorem shows local uniqueness for the families obtained from toroidal crossing spaces whenever a simple section gives the log structure.

\begin{theorem}[Gross--Siebert, Theorem~2.11 in \cite{GrossSiebertII}]\label{locally-unique-defos} 
Let $Y:=X(\log E)\ra S$ be a log toroidal family obtained from a toroidal crossing space $\underline X$ via a simple section $s\in\Gamma(U,\shL\shS_{\underline X})$ and a divisor $E$ as in Lemma~\ref{lemma-add-E}. Let $Y_k$ be a log toroidal deformation over $S_k = \Spec (\NN \to \kk[t]/(t^{k + 1}))$. Then the automorphisms of, isomorphisms of, and obstructions to the existence of a lifting $Y_{k + 1}$ to $S_{k + 1}$ are controlled by $H^0(Y,\Theta^1_{Y/S} \otimes_\kk I)$, $H^1(Y,\Theta^1_{Y/S} \otimes_\kk I)$, and $H^2(Y,\Theta^1_{Y/S} \otimes_\kk I)$ respectively where $I = (t^{k + 1}) \subset \kk[t]/(t^{k + 2})$. In particular, if $V\subset Y$ is affine open, then any two infinitesimal deformations of $V/S$ are isomorphic.
\end{theorem}

\begin{proof}
The proof works precisely as in loc.cit. 
We remark that in Lemma~2.14, the exact sequence in (2) becomes $0\ra \Theta_{Y/S}\ra \Theta_{\underline X/\kk}(\log E)\ra\shB\ra 0$ where $\Theta_{\underline X/\kk}(\log E)$ denotes ordinary derivations that preserve the ideal of $E$. In other words, for the ordinary deformations, we consider the ones of the pair $(\underline X, E)$ rather than just $\underline X$. 
\end{proof}

\section{Differentials for Elementary Log Toroidal Families}
We fix a principal ideal domain $R$ as base ring. The constructions in \S\ref{ElemPretoroid} carry through when replacing $\ZZ$ by $R$. 
We will use the following elementary lemma.
\begin{lemma} \label{lem-intersection-and-wedge}
Let $n,m\ge 0$ and $G_1,...,G_r\subset R^n$ be submodules each of which is a direct summand, then the natural map
$\bigwedge^m_R(\bigcap_i G_i)\ra \bigcap_i\bigwedge^m_R G_i$ is an isomorphism. 
\end{lemma}

First consider the absolute case, i.e.,~an ETD $(Q\subset P,\F)$ with $Q=0$ and let $f:A_{P,\shF}\ra \Spec R$ be the associated log morphism.
One checks that $U$ from \eqref{eq-def-U} is simply the complement of codimension two strata.
Recall from Example~\ref{ex-Danilov} that $W^m:=W^m_{A_{P,\shF}/\Spec R}$ are just the Danilov differentials with log poles in the divisor given by the facets in $\shF$. 
Danilov already computed these in \cite[Proposition~15.5]{Danilov1978} over a field and because of Lemma~\ref{lem-intersection-and-wedge} the same calculation works over $R$ and we obtain the following. 
\begin{proposition}[absolute case] \label{prop-W-abs-case} 
We have a grading
$\Gamma(A_P,W^m)=\bigoplus_{p\in P} (W^m)_p$ with
$$(W^m)_p= \bigwedge^m_R\left(\bigcap_{{F \in \F_{max} \setminus \F}\atop{p \in F}} F^{gp} \otimes_\ZZ R\right)$$
where the intersection is $P^{gp}\otimes_\ZZ R$ if the index set is empty.
\end{proposition}
Let us next assume we have a general ETD $(Q\subset P,\F)$ and let again $f$ denote the associated log toroidal family and $W^m_f:=W^m_{A_{P,\shF}/\Spec A_Q}$ the differentials. 
Note that since $\shF$ contains all vertical facets, every facet in $\F_{\max} \setminus \F$ contains $Q$.
We obtain the following generalization.
\begin{proposition}[general case] 
\label{prop-W-gen-case} 
We have a grading
$\Gamma(A_P,W^m_f)=\bigoplus_{p\in P} (W^m_f)_p$ with
$$(W^m_f)_p= \bigwedge^m_R\left(\left.\left(\bigcap_{{F \in \F_{\max} \setminus \F}\atop{p \in F}} F^{gp} \otimes_\ZZ R\right)\right/ (Q^{gp} \otimes_\ZZ R)\right)$$
where the intersection is $P^{gp}\otimes_\ZZ R$ if the index set is empty. Since $Q^\gp\subset P^\gp$ splits, we can equivalently take the quotient before the intersection.
\end{proposition}
\begin{proof} 
We can compose $f$ with the projection to $\Spec R$ to relate the current situation to that of Proposition~\ref{prop-W-abs-case}. 
The open set $U^{\op{abs}}$ in the absolute case is the complement of $Z^{\op{abs}}$, the union of all codimension two strata.
Hence, $U^{\op{abs}}$ is covered by $U_F$ where $F$ runs over the facets of $P$. 
On the other hand, the open set $U$ for $f$ as given in \eqref{eq-def-U} has a cover $U_F$ where $F$ runs over the essential faces of rank $d-1$ by Lemma~\ref{lem-cover-U}. Obviously, $U^{\op{abs}}\subset U$. 
Note that since $W^m_f$ is locally free on $U$ and $\shO_U$ is $Z^{\op{abs}}$-closed, we find that $W^m_f$ is not only $Z$-closed but also $Z^{\op{abs}}$-closed. Consider the commutative diagram of solid arrows
\begin{equation}
\label{eq-split-exact-abs-to-rel} 
\begin{aligned}
 \xymatrix{
0\ar[r] & f^*\Omega_{A_Q/\Spec R}  \ar^\iota[r]\ar@{=}[d] & W^1_{A_{P,\shF}/\Spec R}\ar[r]\ar[d] & W^1_{f} \ar[d] \ar[r] & 0 \\
0\ar[r] & f^*\Omega_{A_Q/\Spec R} \ar[r]  & W^1_{A_{P}/\Spec R}\ar[r]\ar@{.>}[ul] & W^1_{A_{P}/A_Q} \ar[r]  & 0
}
\end{aligned}
\end{equation}
where the top row is obtained by pushing it forward from $U^{\op{abs}}$.
The bottom sequence is obtained from tensoring the sequence $0\ra Q^\gp\ra P^\gp\ra P^\gp/Q^\gp\ra 0$ with $\shO_{A_P}$, in particular, it is exact and splits. Hence the dotted diagonal arrow exists and commutes with the other maps.
Therefore, $\coker(\iota)$ is a direct summand of $W^1_{A_{P,\shF}/\Spec R}$, in particular $Z^{\op{abs}}$-closed. 
Moreover, $\coker(\iota)\ra W^1_{f}$ is an isomorphism on $U^{\op{abs}}$ and since both sheaves are $Z^{\op{abs}}$-closed, we have $\coker(\iota)=W^1_{f}$ and thus the top row is exact and splits.

Let $\langle f^*\Omega_{A_Q/\Spec R}\rangle$ denote the homogeneous ideal in the sheaf of exterior algebras $W^{\bullet}_{A_{P,\shF}/\Spec R}$ generated by $f^*\Omega_{A_Q/\Spec R}$. The split exactness above gives the split exactness of the following sequence
$$0\ra \langle f^*\Omega_{A_Q/\Spec R}\rangle_m  \ra W^m_{A_{P,\shF}/\Spec R} \ra W^m_{f}\ra 0.$$
Since $A_P$ is affine and $\langle f^*\Omega_{A_Q/\Spec R}\rangle$ coherent, applying $\Gamma(A_P,\cdot)$ to this sequence yields another exact sequence which already gives that $\Gamma(A_P,W^m_f)$ is $P$-graded.
We have 
$\Gamma(A_P,f^*\Omega_{A_Q/\Spec R})= Q^\gp\otimes_\ZZ R[P]$.
Set ${\bf F}_p:=\left(\bigcap_{{F \in \F_{max} \setminus \F}\atop{p \in F}} F^{gp} \otimes_\ZZ R\right)$ and 
let $\langle Q^\gp\otimes R\rangle\subset \bigwedge^\bullet_R {\bf F}_p $ be the homogeneous ideal generated by $Q^\gp\otimes R$.
One computes $\Gamma(A_P,\langle f^*\Omega_{A_Q/\Spec R}\rangle_m)_p=\langle Q^\gp\otimes R\rangle_m$.
Using Proposition~\ref{prop-W-abs-case}, in degree $p\in P$, we obtain the exact sequence
$$0
\ra \langle Q^\gp\otimes R\rangle_m
\ra \bigwedge^m_R {\bf F}_p
\ra  (W^m_f)_p
\ra 0.$$ 
Using a splitting of the injection $(Q^\gp\otimes R)\subset {\bf F}_p$ and comparing leads to the assertion.
\end{proof}

\begin{corollary}\label{cor-genRelFlatW-local}
For all $m$, $W^m_{f}$ is flat over $A_Q$.
\end{corollary}
\begin{proof}
Inspecting the result in Proposition \ref{prop-W-gen-case}, we find $\Gamma(A_P,W^m_f)$ is a free $R[Q]$-module.
\end{proof}

\subsection{Change of Base}\label{ChangeBaseSec}
Let $(Q\subset P,\F)$ be an ETD, $\cT$ be a Noetherian ring and $T = \Spec \cT \to \Spec R[Q]$ be any morphism. Denote by $\sigma$ the composition $Q\ra R[Q]\ra \cT$ which turns $T$ into a coherent log scheme. 
Define $Y$ by the fiber diagram
\begin{equation} 
\label{diagram-define-Y}
\begin{aligned}
 \xymatrix{
Y\ar^c[r]\ar[d] & A_{P,\shF}\ar^f[d]\\
T\ar[r] & A_{Q}
}
\end{aligned}
\end{equation}
of log toroidal families. We want to study when the natural map  $c^*W^m_f\ra W^m_{Y/T}$ is an isomorphism. 
This holds if $f$ is log smooth since then $W^m_f=\Omega^m_f$ are the ordinary log differentials which satisfy this isomorphism property by their universal property. 
In particular, $c^*W^m_f\ra W^m_{Y/T}$ is always an isomorphism on the open set $V:=c^{-1}(U)$. 
The following example shows that it is not an isomorphism in general.
For a subset $I\subset P$, let $\langle I\rangle$ be the smallest face of $P$ containing $I$.

\begin{example}\label{baChaViolation}
Let $P$ be the submonoid of $\ZZ^2$ generated by $(1,0),(1,1),(1,2)$ and let $Q = 0$.
The monoid $P$ has two facets $H_1 = \langle(1,0)\rangle$ and 
$H_2 = \langle(1,2)\rangle$ and setting $\F = \emptyset$ yields an ETD. Let $f:A_{P,\shF}\ra A_Q=\Spec\ZZ$ be the corresponding map.
Now set $\cT=\ZZ/2\ZZ$ inducing the natural map $T=\Spec\cT\ra\Spec\ZZ$ and a fiber diagram as above.
One checks that $c^*W^1_f\ra W^1_{Y/T}$ is not an isomorphism by computing both terms via Proposition~\ref{prop-W-abs-case}.
It suffices to check the degree $p=0$, indeed, $(W^1_f)_0=H_1^\gp\cap H_2^\gp=0$ but
$$(W^1_{Y/T})_0=(H_1^\gp\otimes \ZZ/2\ZZ)\cap (H_2^\gp\otimes \ZZ/2\ZZ)=\ZZ/2\ZZ\cdot (1,0)\subset (\ZZ/2\ZZ)^2.$$
Hence, $((W^1_f)\otimes_\ZZ \ZZ/2\ZZ)_0=0$ but $(W^1_{Y/T})_0\neq 0$.
\end{example}
The example teaches that base change is related to the (non-)commuting of intersection and tensor product. 
The following lemma (that is an elementary exercise in $\Tor$ groups) will help us. 
We say a ring $\cT$ is of \emph{characteristic} $\geq p_0$ if for the residue field $\kappa_{\mathfrak{p}}$ of every point $\mathfrak{p}$ holds 
$\op{char} \kappa_{\mathfrak{p}}\ge p_0$ or $\op{char} \kappa_{\mathfrak{p}}=0$.
\begin{lemma}\label{intBaChaCom}
 Let $G$ be a finitely generated $\ZZ$-module and $H, H' \subset G$ be two submodules. 
 Then there is $p_0$ such that for every ring $\cT$ of characteristic $\geq p_0$ we have 
 $$(H \cap H') \otimes \cT = (H \otimes \cT) \cap (H' \otimes \cT)$$
and each term here is a submodule of $G \otimes \cT$.
\end{lemma}

In the general situation, observe we have 
$\Gamma(Y, \cO_Y) = \bigoplus_{e \in E} z^e \cdot \cT$ with multiplication
$$z^{e_1} \cdot z^{e_2} = z^{e} \cdot \sigma( q) \quad \mathrm{whenever} \quad e_1 + e_2 = e + q$$
with $e\in E,q\in Q$ under the canonical decomposition from \eqref{eq-decomp}.
Similarly, Proposition~\ref{prop-W-gen-case} gives
\begin{equation} \label{eq-describe-cW}
\Gamma(Y, c^*W^m_f) = \bigoplus_{e \in E} z^e \cdot ((W^m_f)_e \otimes_R \cT).
\end{equation}
\begin{lemma} 
Recall $V=c^{-1}(U)$. Equivalent are
\begin{enumerate}
\item the map $c^*W^m_f\ra W^m_{Y/T}$ is an isomorphism,
\item $c^*W^m_f$ is reflexive,
\item the restriction map $\rho: \Gamma(Y, c^*W^m_f) \to \Gamma(V, c^*W^m_f)$ is surjective.
\end{enumerate}
\end{lemma}
\begin{proof} 
\emph{(1)}$\Rightarrow$\emph{(2)}: $W^m_{Y/T}$ is reflexive; 
\emph{(2)}$\Rightarrow$\emph{(3)}: $c^*W^m_f$ is $(Y\setminus V)$-closed;
\emph{(3)}$\Rightarrow$\emph{(1)}: Consider the commutative square
\[
 \xymatrix{
 \Gamma(Y, c^*W^m_f) \ar^\rho[d]\ar[r] & \Gamma(Y, W^m_{Y/T}) \ar[d]\\
 \Gamma(V, c^*W^m_f)\ar[r] & \Gamma(V, W^m_{Y/T})\\
}
\]
where the right vertical map is an isomorphism since $W^m_{Y/T}$ is reflexive by Lemma~\ref{lem-W-reflexive}. 
The bottom horizontal map is an isomorphism by what we said just before Example~\ref{baChaViolation}.
Now \emph{(1)} holds if the top horizontal map is an isomorphism which follows from \emph{(3)} if $\rho$ is additionally injective. This injectivity is a general fact that we prove next. 
Recall that $A_{P,\shF_\max}=A_P$ and we have a map $A_P\ra A_{P,\shF}$ that gives us another commutative square
\begin{equation}
\label{diagram-rho2}
\begin{aligned}
 \xymatrix{
 \Gamma(Y, c^*W^m_f) \ar^\rho[d]\ar[r] & \Gamma(Y, c^*W^m_{A_P/A_Q}) \ar[d]\\
 \Gamma(V, c^*W^m_f)\ar[r] & \Gamma(V, c^*W^m_{A_P/A_Q}).\\
}
\end{aligned}
\end{equation}
Since $A_P\ra A_Q$ is log smooth and $W^m_{A_P/A_Q}=\Omega^m_{A_P/A_Q}$ a free sheaf, the right vertical map is an isomorphism.
We get that $\rho$ is injective if the top horizontal map is injective. 
The latter can be computed from Proposition~\ref{prop-W-gen-case}. 
Indeed, this follows from \eqref{eq-describe-cW} since for every $e\in E$, the cokernel of
$(W^m_f)_e\ra (W^m_{A_P/A_Q})_e$ is a free $R$-module.
\end{proof} 

We next provide a useful criterion for the surjectivity of $\rho$. 
Let $\shE$ be the set of essential faces of $P$ of rank $d-1$. 
By Lemma~\ref{lem-cover-U}, $U$ is covered by $\{U_F|F\in \shE\}$. Set $V_F=c^{-1}(U_F)$ so these cover $V$.
For each $F\in \shE$, choose $e_F\in F$ in the relative interior, i.e., $\langle e_F\rangle =F$.

\begin{theorem}\label{isoCondi}
 Write $M_p := (W^m_f)_p$ for short, and assume that for every subset $\shE'\subset \shE$ and every $e \in E$ the natural map
 $$\left(\bigcap_{F \in \shE'} M_{e + e_F}\right) \otimes_R \cT \to \bigcap_{F \in \shE'} (M_{e + e_F} \otimes_R \cT)$$ 
 is an isomorphism. Then $\rho$ is surjective.
\end{theorem}
\begin{proof} 
We write $M = \Gamma(A_{P}, W^m_f)$, $N = \Gamma(A_{P}, W^m_{A_{P}/A_Q})$ and $N_p$ for the degree $p$ part of $N$. 
By proposition~\ref{prop-W-gen-case}, $M_p$ and $N_p$ only depend on $\langle p \rangle$. We are going to use that for $p_1,p_2\in P$ holds
\begin{equation} \label{eq-face-of-sum}
 \langle p_1+p_2 \rangle=\langle \langle p_1\rangle \cup \langle p_2\rangle \rangle.
\end{equation}
We have a natural injection $M\subset N$ by Proposition~\ref{prop-W-gen-case}. 
Given $\mu \in \Gamma(V, c^*W^m_f)$, we want to show it has a preimage under $\rho$.
We do have a unique preimage $\nu$ under the right vertical map of \eqref{diagram-rho2}, so in $N\otimes_{R[Q]} \cT$ and we are going to show that this preimage lies in $M\otimes_{R[Q]} \cT$. 
Say $\nu = \sum_e z^{e} \cdot n_e $ with $n_e \in N_e \otimes \cT$ is such that $\nu|_V=\mu$.
In particular $\nu|_{V_F}=\mu|_{V_F}$ for all $F\in\shE$.
There is some large $a\ge 1$ so that for each $F\in\shE$
there are $m_{F,e}\in M_e\otimes\cT$ such that
$$\mu|_{V_F} = z^{-a e_F} \sum_e z^e \cdot m_{F,e}$$
and therefore $\nu|_{V_F}=\mu|_{V_F}$ implies
$$z^{a e_F} \sum_e z^e \cdot n_e\ \in\ \bigoplus_{e \in E} z^e \cdot (M_e \otimes_R \cT) \ \subset\ \bigoplus_{e \in E} z^e \cdot (N_e \otimes_R \cT).$$
If $e + ae_F = \tilde e + q$ is the decomposition $P=E\times Q$, then
$n_e \cdot \sigma(q) \in M_{\tilde e} \otimes_R \cT$. 
By \eqref{eq-face-of-sum},
$$e + a e_F \in E \iff \langle e + e_F\rangle\subset E \iff e + e_F \in E,$$
and if this holds, then $\sigma(q) = 1$, so setting 
 $$\shE_e := \{F \in \shE \ | \ e + e_F \in E\}, $$
 we obtain $n_e \in \bigcap_{F \in \shE_e} (M_{e + ae_F} \otimes_R \cT)$ and $M_{e + ae_F}=M_{e + e_F}$. 
Note that $\shE_e$ does not depend on the chosen $e_F$. 
Using the assumption, we get 
 $$n_e \in \bigcap_{F \in \shE_e} (M_{e+e_F} \otimes_R \cT) = \left(\bigcap_{F \in \shE_e} M_{e + e_F}\right) \otimes_R \cT.$$
 For the next step, define
 $\F_e = \{H \in \F_{max} \setminus \F \ | \ \exists F \in \shE_e : e + e_F \in H\}$. 
We use Lemma~\ref{lem-intersection-and-wedge} to compute 

 $$
  \bigcap_{F \in \shE_e} M_{e + e_F} =
 \bigwedge^m_R\left( \bigcap_{H \in \F_e} \frac{H^{gp} \otimes_\ZZ R}{Q^{gp} \otimes_\ZZ R} \right). 
 $$
We finally claim that $\F_e = \{H \in \F_{max} \setminus \F \ | \ e \in H\}$, indeed given an $H$ in the latter, we just need to exhibit an $F\in\shE$ that is also contained in $H$ with $\langle e,F\rangle \subset E$ which can be done since $H\cap E$ is a union of faces in $\shE$.
Thus, $n_e \in M_e \otimes_R \cT$, so indeed $\nu \in M \otimes_{R[Q]} \cT$ and we are done.
\end{proof}

\begin{corollary}\label{locBaChaField}
Let $(Q\subset P,\F)$ be an ETD, $\cT$ a Noetherian ring and $T = \Spec \cT \to A_Q$ a strict morphism of log schemes.
Then $c^*W^m_f$ is reflexive and $c^*W^m_f \to W^m_{Y/T}$ is an isomorphism provided that the composition
$$R\ra R[Q]\ra \cT$$ 
is flat, e.g. when $R$ is a field.
\end{corollary}

As Example~\ref{baChaViolation} shows, the conditions of Theorem~\ref{isoCondi} are not always satisfied in case $R = \ZZ$. 
However, we do get close:

\begin{corollary}\label{locBaChaInteger}
 Let $(Q\subset P,\F)$ be an ETD, and assume $f:A_{P,\F} \to A_Q$ to be defined over $R = \ZZ$. 
Then there is a $p_0 = p_0(Q\subset P,\F)$ such that for every $m$ and every 
$T = \Spec \cT \to A_Q$ with a Noetherian ring $\cT$ of characteristic $\geq p_0$, the sheaf $c^*W^m_f$ is reflexive, and 
 $c^*W^m_f \to W^m_{Y/T}$ is an isomorphism.
\end{corollary}

\begin{proof}
 Applying Lemma~\ref{intBaChaCom} repeatedly, we find for every triple $(m,e,\E')$ as in Theorem~\ref{isoCondi} a $p_0(m,e,\E')$ such that the isomorphism in the theorem holds if $\cT$ is of characteristic $\geq p_0(m,e,\E')$. Since there are only finitely many different sets of modules $\{M_{e + e_F} \ | \ F \in \E'\}$, we obtain one $p_0(Q \subset P,\F)$ that works for all triples.
\end{proof}

For a field $\kk$, consider a monoid ideal $K\subset Q$, let $(K)\subset\kk[Q]$ denote the corresponding monomial ideal of $\kk[Q]$ and set $\cT=\kk[Q]/(K)$.
The map $T=\Spec\cT\ra\kk[Q]$ is the natural one and $Y\ra T$ is defined by \eqref{diagram-define-Y} as before. 
We set $E_K := P \setminus (P + K)$ and note this generalizes the union of essential faces $E$ from \S\ref{ElemPretoroid}, indeed $E=E_{Q\setminus\{0\}}$.
Combining Proposition~\ref{prop-W-gen-case} with Corollary~\ref{locBaChaField} (for $R=\kk$) gives the following which also generalizes \cite[Corollary~1.13]{GrossSiebertII}.
\begin{corollary}\label{relative-on-Y}
$\Gamma(Y, W^m_{Y/T}) \cong \bigoplus_{e \in E_K} z^e \cdot \bigwedge^m \left( \bigcap_{H \in \F_{\max} \setminus \F : e \in H} (H^{gp} \otimes \kk)/(Q^{gp} \otimes \kk) \right)$
with differential $d(z^e \cdot n) = z^e \cdot [e] \wedge n$. 
\end{corollary}
With $c: Y \to A_{P,\F}$ the notation from before, we apply $c^*$ to the split exact sequence given by the top row of \eqref{eq-split-exact-abs-to-rel} and obtain another split exact sequence. 
The left term is free and $c^*W^m_f$ is reflexive by Corollary~\ref{locBaChaField}. Hence, $c^*W^m_{A_{P,\F}/\kk}$ is also reflexive. With $V=c^{-1}(U)$, we find the natural surjection $c^*\Omega^\bullet_{U/\kk}\sra\Omega^\bullet_{V/\kk}$ to be an isomorphism (e.g. by local freeness of both).
For $j:V\hra Y$ the inclusion and $W^\bullet_Y := j_*\Omega^\bullet_{V/\kk}$ we thus have $c^*W^m_{A_{P,\F}/\kk} \cong W^m_Y$. Plugging this into Proposition~\ref{prop-W-abs-case} yields the following.
\begin{corollary}\label{absolute-on-Y}
$\Gamma(Y, W^m_{Y}) \cong \bigoplus_{e \in E_K} z^e \cdot \bigwedge^m \left( \bigcap_{H \in \F_{\max} \setminus \F : e \in H} H^{gp} \otimes \kk \right)$
with differential $d(z^e \cdot n) = z^e \cdot e \wedge n$.
\end{corollary}

\subsection{Local Analytic Theory} \label{subsec-local-analytic}
We keep the setup and notation from before (with $\kk = \CC$), so $(Q\subset P,\F)$ is an ETD and $K\subset Q$ a monoid ideal. We additionally assume that $Q \setminus K$ is finite, so $\cT=\CC[Q]/(K)$ is an Artinian local ring.
For $P^+=P\setminus\{0\}$, let $\CC\llbracket P\rrbracket$ be the completion of $\CC[P]$ in $(P^+)$.
\begin{lemma}(\cite[Proposition\,V.1.1.3]{LoAG18}) \label{lem-ogus-analytic}
For every local homomorphism $h: P \to \NN$, i.e., $h^{-1}(0)=\{0\}$ and we may view $h$ as a grading, it holds
$$\cO_{A_P^{an},0} = \left\{\left.\sum_{p \in P} \alpha_p z^p \ \right|\ \alpha_p\in\CC,\,\mathrm{sup}_{p \in P^+} \left\{ \frac{\mathrm{log}\ |\alpha_p|}{h(p)} \right\} < \infty \right\} \subset \CC\llbracket P\rrbracket.$$
\end{lemma}
We have $\Gamma(Y,\cO_Y) \cong \CC[E_K] := \bigoplus_{e \in E_K}\CC \cdot z^e$ with $z^{e} \cdot z^{e'} = z^{e + e'}$ if $e + e' \in E_K$ and $z^e \cdot z^{e'} = 0$ otherwise. 
By \cite[Corollary~3.2]{SC_fiberproduct_61} and Lemma~\ref{lem-ogus-analytic}, the complete local ring at the origin in $Y^{an}$ is
$$
\hat \cO_{Y,0} \cong (\CC[Q]/(K)) \otimes_{\CC\llbracket Q\rrbracket} \CC\llbracket P\rrbracket \cong \left \{\sum_{e \in E_K} \alpha_e z^e\right\}=:\CC\llbracket E_K\rrbracket.
$$
Lemma~\ref{lem-ogus-analytic} together with the surjectivity of $\cO_{A_P^{an},0} \to \cO_{Y^{an},0}$ and Krull's intersection theorem yields
\begin{equation}\label{eq-stalk-Y-analytic}
\cO_{Y^{an},0} = \left\{\left.\sum_{e \in E_K} \alpha_e z^e \in \CC\lfor E_K\rfor \ \right|\ \mathrm{sup}_{e \in E_K \setminus 0} \left\{ \frac{\mathrm{log} |\alpha_e|}{h(e)} \right\} < \infty \right\}.
\end{equation}

\begin{lemma}\label{ana-stalk}
 Let $(V,\langle \cdot , \cdot \rangle)$ be a finite-dimensional $\CC$-vector space with a Hermitian inner product. 
 For every $e \in E_K$, let $V_e \subset V$ be vector subspaces so that $$\tilde V := \bigoplus_{e \in E_K} z^e \cdot V_e \subset V[E_K]$$
 is a $\CC[E_K]$-module. Assume moreover that $V_e \subset V$ depends only on the set $F(e) := \{H \subset P \ \mathrm{a \ facet} \ | \ Q \subset H, e \in H\}$. 
 Set $V\lfor E_K \rfor := \prod_{e \in E_K} z^e \cdot V_e$ and $\shV^{an}:=\tilde V\otimes_{\CC[E_K]} \cO_{Y^{an}}$.
 We find its stalk at the origin to be
 $$\shV^{an}_0 \cong \left\{\left. \sum_{e \in E_K}  z^e \cdot v_e \in V\lfor E_K \rfor  \ \right|\ \mathrm{sup}_{e \in E_K \setminus 0} \left\{ \frac{\mathrm{log}\ \| v_e\| }{h(e)} \right\} < \infty \right\}.$$
\end{lemma}
\begin{proof} The set of possible $F(e)$ is finite, so there is only a finite set of different $V_e$. Choosing orthonormal bases for all $V_e$ allows to reduce the assertion to \eqref{eq-stalk-Y-analytic}. We leave the technical details to the reader.
\end{proof}

\begin{remark} \label{rem-compute-anal-stalk}
We can use Lemma~\ref{ana-stalk} to compute the stalk at $0$ of the analytification of $W^m_{Y/T}$ and $W^m_{Y}$ by using Corollary~\ref{relative-on-Y} and Corollary~\ref{absolute-on-Y} respectively.
\end{remark}

\section{Base Change of Differentials for Log Toroidal Families}\label{TorBaCha}

\begin{definition}[BC]
We say that a generically log smooth morphism $f:X\ra S$ satisfies the \emph{basechange property} if for every strict
morphism $T\ra S$ of Noetherian fs log schemes, $m\in\ZZ$ and $c$ the map given by the Cartesian diagram
\[
 \begin{CD} 
  Y @>c>> X \\
  @VgVV @VfVV \\
  T @>b>> S, \\
 \end{CD}
\tag{BC}\label{BC}
\]
the sheaf $c^*W^m_{X/S}$ is reflexive or equivalently, the natural map $c^*W^m_{X/S} \to W^m_{Y/T}$ is an isomorphism.
\end{definition}

\begin{theorem}[Base Change over Fields] \label{baseChangeField}
Let $f: X \to S$ be a log toroidal family over a field $\kk$, then $f$ satisfies \eqref{BC}.
\end{theorem}
\begin{proof} 
This follows directly from the local statement Corollary~\ref{locBaChaField}.
\end{proof}

\begin{theorem}[Generic Base Change]\label{baseChangeGeneric}
Let $f: X \to S$ be a log toroidal family. 
Then there is a finite set of prime numbers $p_1, ..., p_N \in \ZZ$ so that if 
$f^\circ: X^\circ \to S^\circ$ is obtained from $f$ by inverting $p_1,...,p_N$ (i.e., base change to $\Spec\ZZ_{p_1...p_N}$), then
$f^\circ$ satisfies \eqref{BC}.
\end{theorem}
\begin{proof}
Again, this follows directly from the local statement Corollary \ref{locBaChaInteger} combined with the fact that we can use a finite cover by local models. 
\end{proof}
An application of the above theorems is the following lemma which is crucial for the degeneration of the Hodge--de Rham spectral sequence.
\begin{lemma}\label{deriBaCha}
 (cf.~\cite[Proposition~6.6]{Illusie2002a}) Let $f: X \to S$ be a proper log toroidal family with $S$ affine, and let $b: T \to S$ with $T$ affine. Assume $c^*W^m_{X/S} = W^m_{Y/T}$ holds for all $m$. Then we have isomorphisms
 \begin{align}
  Lb^*Rf_*W^p_{X/S} &\to Rg_*W^p_{Y/T} \label{BC-iso1}  \\
  Lb^*Rf_*W^\bullet_{X/S} &\to Rg_*W^\bullet_{Y/T} \label{BC-iso2}  
 \end{align}
 in $D^b(T)$. If, for fixed $p$, all $R^qf_*W^p_{X/S}$ are locally free of constant rank, then \eqref{BC-iso1} induces an isomorphism 
 $$b^*R^qf_*W^p_{X/S} \xrightarrow{\cong} R^qg_*W^p_{Y/T}.$$
 If, for all $n$, the sheaf $R^nf_*W^\bullet_{X/S}$ is locally free of constant rank, then \eqref{BC-iso2} induces an isomorphism 
 $$b^*R^nf_*W^\bullet_{X/S} \xrightarrow{\cong} R^ng_*W^\bullet_{X/S}.$$
\end{lemma}
\begin{proof}
Since $W^m_{X/S}$ is flat over $S$--- this is Corollary~\ref{cor-genRelFlatW-local}---, the proof becomes identical to that in \cite[Proposition~6.6]{Illusie2002a}. 
\end{proof}

\section{Spreading Out Log Toroidal Families}\label{TorSpreadOut}
We fix a sharp toric monoid $Q$, a field $\kk \supset \QQ$ and set $S = \Spec (Q \to \kk)$ where the map $Q \to \kk$ is $q\mapsto 0$ except $0\mapsto 1$. 
We choose distinct subrings $B_\lambda\subset \kk$ for all $\lambda$ in some index set $\Lambda$ so that any two 
$B_{\lambda_1}, B_{\lambda_2}$ are both contained in a third $B_{\lambda}$. 
We say $\lambda_1\le\lambda_2$ if $B_{\lambda_1}\subset B_{\lambda_2}$.
Furthermore, we require $\varinjlim_\lambda B_\lambda = \kk$ and that each $B_\lambda$ is of finite type over $\ZZ$.
We get log schemes $S_\lambda = \Spec (Q \to B_\lambda)$ each with a strict map $S\ra S_\lambda$ and in fact $S=\liminv_\lambda S_\lambda$.

\begin{proposition}\label{spreadToroidal}
Let $f: X \to S$ be a log toroidal family of relative dimension $d=\rk{\Omega^1_{U/S}}$. 
Then there is $\lambda \in \Lambda$ and a log toroidal family $f_\lambda: X_\lambda \to S_\lambda$, so that $f$ is obtained by base change from $f_\lambda$, i.e., there is a Cartesian square
\[
 \begin{CD}
  X @>>> X_\lambda \\
  @VfVV @VV{f_\lambda}V \\
  S @>>> S_\lambda. \\
 \end{CD}
\]
If $f$ is separated and/or proper, we can assume $f_\lambda$ to be so, too.
\end{proposition}

\begin{proof}
 By \cite[Theorem~8.8.2 (ii)]{Grothendieck1966}, \cite[Theorem~8.10.5]{Grothendieck1966} and \cite[Theorem~11.2.6 (ii)]{Grothendieck1966} we can find a $\lambda \in \Lambda$ and a morphism $f_\lambda: X_\lambda \to S_\lambda$ that is finitely presented and flat, and an isomorphism $S \times_{S_\lambda} X_\lambda \cong X$ over $S$. 
If $f: X \to S$ is separated respective proper, we can choose $f_\lambda$ moreover separated respective proper. 
Using \cite[Corollaire~12.1.7(iii)]{Grothendieck1966} and \cite[Theorem~8.10.5]{Grothendieck1966}, we can choose $\lambda$ 
 such that $f_\lambda$ is a Cohen--Macaulay morphism. Since these decompose disjointly over the relative codimension, again by increasing $\lambda$ if needed, we may assume that $f_\lambda$ has relative dimension $d$.

We next spread out $U$ such that $U_\lambda \subset X_\lambda$ satisfies \eqref{CC}. 
We do this by spreading out its complement $Z$. Indeed, by \cite[05M5, Lemma~31.16.1]{stacks}, we can increase $\lambda$ so that every fiber of $Z_\lambda \to S_\lambda$ has dimension $\leq d - 2$ and then define $U_\lambda:=X_\lambda\setminus Z_\lambda$.

Now a straightforward generalization of the method of \cite[Lemma~4.11.1]{Tsuji1999} yields that, for appropriate $\lambda$, we can find a log structure on $U_\lambda$ and upgrade $f_\lambda$ to a log morphism such that $U_\lambda$ is fs and $f_\lambda$ is log smooth and saturated. More precisely, we choose an affine \'etale cover $\{U_i\}_i$ of $U$ such that $U_i \to S$ admits a local model by a saturated morphism $\theta_i: Q \to P_i$ of monoids; the local model remains a local model for an appropriate spread-out $U_{i,\lambda} \to S_\lambda$ (for appropriate $\lambda$), which thus carries the structure of a saturated log smooth morphism to $S_{\lambda}$. At this point, the $U_{i,\lambda}$ might not cover $U_\lambda$, the log structures on $U_{i,\lambda}$ might not coincide on overlaps, and even if this was the case, we might not have a log morphism to $S_\lambda$. We achieve all this by increasing $\lambda$.
 
Finally---again by possibly increasing $\lambda$---we show that the family $f_\lambda: X_\lambda \to S_\lambda$ is log toroidal.
We fix a finite covering $\{V_i \to X\}$ with local models $(Q\subset P_i, \F_i)$ as in Definition \ref{toroidDef}, and 
for each of them, we construct a diagram as in Figure~\ref{toroidal-figure} below. 
Namely we first spread out $V_i \to S$ to $V_{i,\lambda} \to S_\lambda$. Then $L_{i,\lambda}$ is defined 
by base change, and we construct the \'etale morphisms of schemes 
$g_\lambda: V_{i,\lambda} \to X_\lambda$ and $h_\lambda: V_{i,\lambda} \to L_{i,\lambda}$ also by spreading out. We can assume that $X_\lambda$ is covered by $\{V_{i,\lambda} \to X_\lambda\}$
and that $\tilde U_i \subset V_i$ spreads out to an open $\tilde U_{i,\lambda} \subset V_{i,\lambda}$ satisfying \eqref{CC}. We get two log structures $(g_\lambda)^*_{log}\M_{X_\lambda}$ and 
$(h_\lambda)^*_{log}\M_{L_{i,\lambda}}$ on $\tilde U_{i,\lambda}$ which we identify by 
\cite[Sublemma~4.11.3]{Tsuji1999}. By the same sublemma, the two morphisms 
$(g \circ f)^*_{log}\M_{S_\lambda} \to \M_{\tilde U_{i,\lambda}}$ coming from 
$f_\lambda \circ g_\lambda$ respective $r_\lambda \circ h_\lambda$ coincide. Since $\{V_i \to X\}$ is a 
finite covering, we can find $\lambda$ that admits the above construction 
for all $V_i$ simultaneously. 
\end{proof}
\begin{figure}
 \[
 \xymatrix@R-1pc@C-1pc{
   & V_i \ar[dl]^<<{\!\!\!\!g} \ar[rr] \ar@{.>}[dd]^<<<<{\!h} & & V_{i,\lambda} \ar@{.>}[dd]^<<<<{\!h_\lambda} 
   \ar[dl]^<{\!\!\!\!g_\lambda} & & \\ 
   X \ar[dd]^<<<<{\!f} \ar[rr]_<<<<{p_\lambda} & & X_\lambda \ar[dd]^<<<<{\!f_\lambda} & & & \\
   & L_i \ar[dl]^<<{\!\!\!\!r} \ar[rr] & & L_{i,\lambda} \ar[dl]^<{\!\!\!\!r_\lambda} \ar[rr] & & 
   A_{P_i,\F_i} \ar[dl] \\
   S \ar[rr]_<<<<{q_\lambda} & & S_\lambda \ar[rr]_<<<<{a_\lambda} & & A_Q. & \\
 }
\]
\caption{The diagram constructed in the text.}
\label{toroidal-figure}
\end{figure}

\section{The Cartier Isomorphism}\label{TorCarIso}

In this section, we define the Cartier homomorphism for a generically log smooth family $f: X \to S$ in characteristic $p>0$. 
We then prove that it is an isomorphism if $f$ is log toroidal.
Similar to \cite{Blickle2001}, we first study the situation on $U$ and then examine its extension to all of $X$. 
Let $F_S: S \to S$ be the absolute log Frobenius on the base, i.e., given by taking $p$th power in $\shM_S$ and $\shO_S$ respectively, we similarly define $F_X:X\ra X$. We define $f':X'\ra S$ and the relative Frobenius $F$ by the Cartesian square
\[
 \xymatrix{
 X\ar_f[dr]\ar^(.6)F[r]\ar@/^{1.5pc}/^{F_X}[rr]&X' \ar^{f'}[d]\ar^(.45)s[r] & X\ar^f[d]\\
 &S \ar^{F_S}[r] & S.
 }
\]
Set $U':=s^{-1}(U)$ and $Z'=X'\setminus U'$. 
\begin{theorem}[\cite{kkatoFI}] \label{thm-kato-Cartier}
We have a canonical (Cartier) isomorphism of $\shO_{U'}$-modules
$$C^{-1}_{U} : \Omega^m_{U'/S} \to \cH^m(F_{*}\Omega^\bullet_{U/S})$$
which is compatible with $\wedge$ and satisfies $C^{-1}(a) = F^*(a)$ for $a \in \cO_{X'}$ and\newline $C^{-1}(\dlog(s^*q)) = \dlog(q)$ for $q \in \M_U$.
\end{theorem}

\begin{proof} 
This is \cite[Theorem~4.12(1)]{kkatoFI} once we identify $U''=U'$: Kato considers the factorization $U\stackrel{g}\ra U''\stackrel{h}\ra (U')^{\op{int}}\stackrel{i}\ra U'$ of $F|_U$ where $i$ is the integralization of $U'$ and $g\circ h$ is the unique factorization of this weakly purely inseparable morphism where $h$ is \'etale and $g$ purely inseparable, using \cite[Proposition~4.10(2)]{kkatoFI}.
Now $i$ is an isomorphism because $f$ is integral. 
By \cite[Corollary~III.2.5.4]{LoAG18}, since $f: U \to S$ is saturated, $F:U\ra U'$ is exact. 
The uniqueness of the factorization $g\circ h$ now implies that $h$ is an isomorphism.
\end{proof}

Since $W^m_{X'/S}$ is $Z'$-closed, pushing forward the inverse of $C^{-1}_{U}$ to $X'$, we obtain a homomorphism 
$$C: \cH^m(F_{*}W^\bullet_{X/S}) \to W^m_{X'/S}$$ 
which is an isomorphism on $U'$. We obtain the following lemma. 
\begin{lemma} \label{lem-C-iso-iff-Zclosed}
The map $C$ is an isomorphism if and only if $\cH^m(F_{*}W^\bullet_{X/S})$ is $Z'$-closed.
\end{lemma}
\begin{definition} \label{def-cartier-iso}
We say that a generically log smooth family $f:X\ra S$ in positive characteristic has the \emph{Cartier isomorphism property} if
$C$ is an isomorphism for all $m\ge 0$.
\end{definition}

By Theorem~\ref{thm-kato-Cartier}, $\cH^m(F_{*}W^\bullet_{X/S})$ is locally free on $U'$, hence it is $Z'$-closed if and only if it is reflexive. Reflexivity can be checked \'etale locally. 

\begin{lemma}\label{explCarIso}
 Let $(Q\subset P,\F)$ be an ETD, let $b: T \to A_Q$ be strict with $\underline T = \Spec \cT$ and consider the Cartesian diagram
\[
 \xymatrix{
Y\ar^c[r]\ar_g[d] & A_{P,\shF}\ar^f[d]\\
T\ar^b[r] & A_Q.
 }
\]
Then $\cH^m(F_*W^\bullet_{Y/T})$ is reflexive.
\end{lemma}
\begin{corollary}\label{torCartierIso}
 Every log toroidal family $f: X \to S$ over $\FF_p$ has the Cartier isomorphism property.
\end{corollary}

\begin{proof}[Proof of Lemma \ref{explCarIso}]
Set $V := c^{-1}(U_{P})$ and let $Y', V'$ be the base changes by the absolute Frobenius $F_T$. Let $F: Y \to Y'$ be the relative Frobenius. 
Inspired by the Frobenius decomposition \cite[Theorem~2.1]{Deligne1987}, we construct a homomorphism 
 $\phi^\bullet : \bigoplus_{m} W^m_{Y'/T}[-m] \to F_*W^\bullet_{Y/T}$
 of \emph{complexes of $\cO_{Y'}$-modules} which induces an isomorphism in cohomology. 
Since the left hand side has zero differentials, the assertion then follows from the reflexivity of $W^m_{Y'/T}$ given by Lemma~\ref{lem-W-reflexive}.

Similar to \S\ref{ChangeBaseSec}, we find explicitly that 
$R' := \Gamma(Y',\cO_{Y'}) = \bigoplus_{e \in E} z^e \cdot \cT$ with 
$$z^{e_1} \cdot z^{e_2} = z^e \cdot \sigma(q)^p \quad \mathrm{whenever} \quad e_1 + e_2 = e + q$$
with $e \in E, q \in Q$. We have $s^*(z^e \cdot t) = z^e \cdot t^p$ and $F^*(z^e \cdot t) = z^{p \cdot e} \cdot t$. After writing $W^m_e := (W^m_f)_e \otimes_{\FF_p} \cT$, 
the module $\Gamma(Y', W^m_{Y'/T})$ is given by the $\cT$-module $\bigoplus_{e \in E} z^e \cdot W^m_e$
on which $R'$ acts as
$$(z^{e_1} \cdot t_1) \cdot [z^{e_2} \cdot (w \otimes t_2)] = z^{e} \cdot (w \otimes \sigma( q)^pt_1t_2) \quad \mathrm{whenever} \quad e_1 + e_2 = e + q$$
with $e \in E, q \in Q$.
Similarly, $\Gamma(Y',F_*W^m_{Y/T})$ is given by the same $\cT$-module, however now $R'$ acting via $F^*$ as 
$$(z^{e_1} \cdot t_1) \cdot [z^{e_2} \cdot (w \otimes t_2)] = z^{e} \cdot (w \otimes \sigma(q)t_1t_2) \quad \mathrm{whenever} \quad p \cdot e_1 + e_2 = e + q.$$
Note the subtle difference. The differential on $F_*W^\bullet_{Y/T}$ is given by $d(z^e \cdot (w \otimes t)) = z^e \cdot ([e] \wedge w \otimes t)$. We define
$$\phi^\bullet : \bigoplus_{m} W^m_{Y'/T}[-m] \to F_*W^\bullet_{Y/T}, \quad z^e \cdot (w \otimes t) \mapsto z^{p \cdot e} \cdot (w \otimes t),$$
and claim $\cH^m(\phi^\bullet)$ is an isomorphism. 
Indeed, first note that $\phi^\bullet$ itself is injective. 
Then set $E_p = \{p \cdot e | e \in E\}$. 
We have 
$\mathrm{im}(\phi^m) = \bigoplus_{e \in E_p} z^e \cdot W^m_e$ because $W^m_e=W^{m}_{e/p}$ for $e\in E_p$ by Proposition~\ref{prop-W-gen-case}.
Denoting the coboundaries of
$F_*W^m_{Y/T}$ by $B^m$, we have $\mathrm{im}(\phi^m) \cap B^m = 0$ since  
$0 = [e] \in W^1_e$ for $e \in E_p$ because $e=pe'$ and $p$ is zero in $\shT$. This readily gives that $\cH^m(\phi^\bullet)$ is injective.
For surjectivity, if $e \notin E_p$, observe that $[e] \not= 0$, so if $w \in W^m_e$, then 
 $[e] \wedge w = 0$ if and only if there is some $w' \in W^{m - 1}_e$ with $[e] \wedge w' = w$.
\end{proof}
\begin{remark}
 We believe that $\cH^m(\phi^\bullet)$ is the log Cartier isomorphism on $V'$.
\end{remark}

\section{The Decomposition of $F_*W^\bullet_{X_0/S_0}$}\label{TorDecompo}

We prove a log version of the decomposition theorem \cite[Theorem~2.1]{Deligne1987} in the setting of generically log smooth families. (We noticed that \cite[Corollary~3.7]{Deligne1987} alias \cite{Illusie2002a} does not generalize well to the generically log smooth setting.) 
The assumption for $f: X \to S$ to be saturated on the log smooth locus allows a simpler approach than \cite[Theorem~4.12]{kkatoFI}. 
Our setting is as 
follows: let $k$ be a perfect field with $\mathrm{char} \ k = p$ (thus $\ZZ/p^2\ZZ\ra W_2(k)$ is flat), and let $Q$ be a sharp toric monoid. 
Set $S_0 = \Spec (Q \to k)$ and $S = \Spec (Q \to W_2(k))$ where in both cases $Q \ni q \mapsto 0$ except $0\mapsto 1$. 
The Frobenius endomorphism on $k$ becomes an endomorphism $F_0$ of $S_0$ via $Q\ni q\mapsto pq$.
Similarly, its lift to $W_2(k)$ defined via $(a_1, a_2) \mapsto (a_1^p, a_2^p)$ becomes\footnote{Warning: This is \emph{not} the $p$th power map on $W_2(k)$ and thus depends on the chosen chart.} an endomorphism $F_S$ of $S$ that restricts to $F_0$ on $S_0$.
Let $f: X \to S$ be a generically log smooth family and let $f_0: X_0 \to S_0$ be its restriction to $S_0$. We consider the commutative diagram of generically log smooth families as in Figure~\ref{Frobenius-diagram},
where $X'_0, X'$ are defined by requiring the front and back square to be Cartesian and $F$ is the relative Frobenius, i.e., $F$ is induced by the back square's Cartesianness using the Frobenius endomorphisms on $X_0$ and $S_0$.
Since $X$ does not have a Frobenius, we do \emph{not} easily obtain the dotted arrow $G$ in a similar way and in general it does not exist globally. We call a locally defined morphism $G$ that fits into the 
diagram a \emph{local Frobenius lifting}. Because the (Zariski or \'etale) topologies are identified 
along $F$ and $i$, we can define Frobenius liftings simply at the level of sheaves:
\begin{definition}
 Let $Y' \to X'$ be an \'etale open. Then a Frobenius lifting $G: Y \to Y'$ on $Y'$ consists of a ring homomorphism $G^*: \cO_{Y'} \to G_*\cO_Y$ yielding a morphism of schemes and a monoid homomorphism 
 $G^*: \M_{Y'}|_{V'} \to G_*\M_Y|_{V'}$ defined on some $V' \subset Y'$ satisfying $(CC)$, yielding a log morphism. Two Frobenius liftings are considered equal if they are equal on some smaller (Zariski) open satisfying $(CC)$. 
The Frobenius liftings form an \'etale sheaf of sets $\F rob(X,X')$.
\end{definition}
\begin{figure}
 
\[
 \xymatrix{
  & X_0 \ar[rr]^{F} \ar@/_1pc/[ddrr]_{f_0}|(.59)\hole \ar[ld]_i & & X'_0 \ar[dd]|\hole^(.3){f_0'} \ar[rr]^{s}  \ar[ld]_{i'} & & X_0 \ar[dd]^{f_0}  \ar[ld] & & \\
  X \ar@/_1pc/[ddrr]_f \ar@{.>}[rr]^G & & X' \ar[rr] \ar[dd]^{f'} & & X \ar[dd]^(.3)f & & & \\
  & & & S_0 \ar_(.3){F_0}|\hole[rr]  \ar[ld] & & S_0 \ar@{}[r]^(.08){}="a"^(.94){}="b"^(1.26){}="A" \ar "a";"b" \ar[ld]  & \hspace{1cm} {\Spec\FF_p} \\
  & & S \ar[rr]^{F_S} & & S \ar@{}[rr]^(.02){}="a"^(.63){}="b"^(.82){}="B" \ar "a";"b"  & & \hspace{-1cm}{\Spec\ZZ/p^2\ZZ}  
\ar "A"-<10pt,10pt>;"B"+<10pt,10pt>{}
 }
\]
\caption{The diagram.}
\label{Frobenius-diagram}
\end{figure}
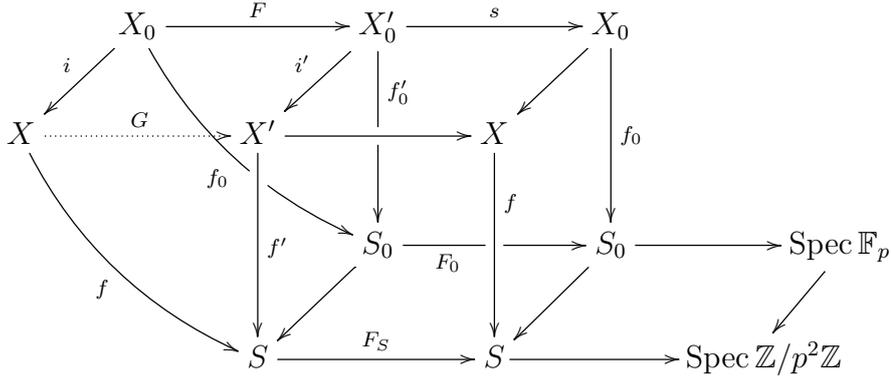

\begin{remark}
 We need the flexibility of $V'$ in the definition of $\F rob(X,X')$ to construct Frobenius liftings from local models as they occur for log toroidal families. We will see below that we could have as well required the log 
 part to be defined on $Y' \cap U'$, see the proof of Proposition \ref{frobLiftToQuiso}.
\end{remark}

Let $j:U'\hra X'$ denote the pullback of $U\subset X$ and $Z'=X'\setminus U'$.
By Lemma~\ref{relaNormal}, $\F rob(X,X') = j_*(\F rob(X,X')|_{U'})$.
Let $\cI \subset \cO_X$ be the ideal sheaf defining $X_0 \subset X$, flatness gives
$\cI = p \cdot \cO_X \cong \cO_{X_0}$. Using $\cI^2=0$, one checks that $F_*\cI$ is an $\shO_{X'}$-module.
Considering derivations on $U'$ with values in $F_*\cI$, we obtain a sheaf of groups $\G := j_*\shD er_{U'/S}(F_*\cI)=j_*\shH om(\Omega^1_{U'/S},F_*\cI)$ which agrees with $\shH om(W^1_{X'/S},F_*\cI)$ because $F_*\cI$ is $Z'$-closed by Lemma~\ref{relaNormal}.


\begin{lemma}
The restriction $\F rob(X,X')|_{U'}$ is a $\G|_{U'}$-torsor; hence $\F rob(X,X')$ is a $\G$-pseudo-torsor.
\end{lemma}
\begin{proof}
 Let $\shD$ be the sheaf of sets on $U'$ given by \'etale local deformations of the diagram 
\[
 \begin{CD}
  U_0 @>i'\circ F>> U' \\
  @ViVV @Vf'VV \\
  U @>f>> S \\
 \end{CD}
\]
in the sense of \cite[Definition~IV.2.2.1]{LoAG18}, i.e., $\shD$ is the sheaf of morphisms $U \to U'$ making the diagram commute. The sheaf $\shD$ is a $\G|_{U'}$-pseudo-torsor by \cite[Theorem~IV.2.2.2]{LoAG18}
and because $f': U' \to S$ is smooth, 
it is a torsor. Because $\Omega^1_{U'/S}$ is locally free, $\shD$ is locally isomorphic to $(F_*\cI)^{\oplus d}$.
By Lemma~\ref{relaNormal}, $\shD$ is $\tilde Z$-closed for every $\tilde Z \subset X'$ satisfying $\codim{\tilde Z}{X'}\ge 2$. 
By this property, the obvious homomorphism $\shD \to \F rob(X,X')|_{U'}$ is an isomorphism of sheaves of sets making $\F rob(X,X')|_{U'}$ a $\G|_{U'}$-torsor. 
\end{proof}

\begin{proposition}\label{frobLiftToQuiso}
 Let $Y' \to X'$ be an \'etale open and $G: Y \to Y'$ a local Frobenius lifting.
Then there is a canonical homomorphism of complexes 
$$\phi_G: W^1_{Y_0'/S_0}[-1] \to F_*W^\bullet_{Y_0/S_0}$$
inducing the Cartier isomorphism in first cohomology on $U_0' \cap Y_0'$. If $h \in \G(Y')$, then 
$\phi_G$ and $\phi_{h \cdot G}$ are related by
$$\phi_{h \cdot G} = \phi_G + (F_*d) \circ \tilde h$$
where $\tilde h: W^1_{Y_0'/S_0} \to F_*\cI \cong F_*W^0_{Y_0/S_0}$ is the induced homomorphism.
\end{proposition}
\begin{proof}
We choose $V'=U'\cap Y'$ for the representative of $G$.
The straightforward log version of the construction of \cite[Proposition~3.8]{Illusie2002a} yields a homomorphism 
 $\Omega^1_{V'_0/S_0} \to F_*\Omega^1_{V_0/S_0}$ and this has also been used implicitly by Kato in 
 \cite[Theorem~4.12]{kkatoFI}. Applying $j_*$ yields $(\phi_G)^1$, and we define the other $(\phi_G)^m$ 
 to be $0$. The resulting $\phi_G$ does not depend on $V'$ since the involved sheaves 
are $\tilde Z$-closed for every $\tilde Z \subset Y_0'$ satisfying $\codim{\tilde Z}{Y_0'}\ge 2$, so $\phi_G$ is well-defined. 
The construction yields that $\cH^1(\phi_G)$ is the 
 Cartier isomorphism of Theorem~\ref{thm-kato-Cartier} on $V'_0=U_0' \cap Y_0'$.
The second statement is similar to \cite[Lemma~5.4,(5.4.1)]{Illusie2002a} except that we use the more elegant language of torsors (as already remarked in \cite[Remark\,2.2\,(iii)]{Deligne1987}) which renders the analog of \cite[Lemma~5.4,(5.4.2)]{Illusie2002a} trivial.
\end{proof}


\begin{theorem}\label{genDecompThm}
 Let $f: X \to S$ be a generically log smooth family, assume that $f_0: X_0 \to S_0$ has the Cartier isomorphism property (Definition~\ref{def-cartier-iso}), and assume that $\F rob(X,X')$ is a $\G$-torsor.
 Then we have a quasi-isomorphism
 $$\bigoplus_{m < p} W^m_{X_0'/S_0}[-m] \to \tau_{<p}F_*W^\bullet_{X_0/S_0}$$
 in $D^b(X_0')$ where $\tau_{<p}$ means the truncation of a complex. 
\end{theorem}
\begin{proof}
 Because $\F rob(X,X')$ is a torsor, we can find an \'etale cover $\Y = \{Y_{\alpha}'\}$ of $X'$ such that we have a local Frobenius lifting $G_{\alpha}: Y_{\alpha} \to Y_{\alpha}'$. We obtain an induced cover $\Y_0$ of $X_0'$.
 On the log smooth locus $U_0' \subset X_0'$, we can apply an argument as implicitly used in \cite[Theorem~4.12]{kkatoFI}:
 using Proposition \ref{frobLiftToQuiso}, the gluing method of Step \textbf{B} in the proof of \cite[Theorem~5.1]{Illusie2002a} yields a 
 homomorphism 
 $$\varphi: \Omega^1_{U_0'/S_0}[-1] \to \check \C^\bullet(\Y_0 \cap U_0', F_*\Omega^\bullet_{U_0/S_0}) =: \check \C^\bullet_{U}$$
 of complexes of sheaves where $\check\C^\bullet(\mathfrak{U},\shF^\bullet)$ refers to the total sheaf \v{C}ech complex for a cover $\mathfrak U$ and a complex of sheaves $\shF^\bullet$. 
We also have the natural quasi-isomorphism $$\psi: F_*W^\bullet_{X_0/S_0} \to \check \C^\bullet(\Y_0,F_*W^\bullet_{X_0/S_0}).$$
Using $\psi$ and that the question is local, Proposition~\ref{frobLiftToQuiso} gives that $\varphi$ induces the Cartier isomorphism on $U_0'$ for $\shH^1$. 
Now let $0 \leq m < p$. With the antisymmetrization map $a_m: \Omega^m_{U_0'/S_0}[-m] \to (\Omega^1_{U_0'/S_0}[-1])^{\otimes m}$ 
 defined by $a_m(\omega_1\wedge...\wedge\omega_m)=\frac1{m!}\sum_{\sigma\in S_m} \op{sgn}(\sigma)\omega_{\sigma(1)}\otimes...\otimes\omega_{\sigma(m)}$, we obtain a morphism 
 $$\varphi^m: \Omega^m_{U_0'/S_0}[-m] \xrightarrow{a_m} (\Omega^1_{U_0'/S_0}[-1])^{\otimes m} \xrightarrow{\varphi^{\otimes m}} (\check \C^\bullet_U)^{\otimes m} \to \check\C^\bullet_U$$
 where the last map is induced by the wedge product on $F_*\Omega^\bullet_{U_0/S_0}$. Note that the various $\varphi^m$ are compatible with the wedge product of $\Omega^\bullet_{U_0'/S_0}$ and of the cohomology of 
 $F_*\Omega^\bullet_{U_0/S_0}$ hence $\varphi^m$ induces the Cartier isomorphism in cohomology. Taking the sum, we obtain a quasi-isomorphism
 $$\varphi^\bullet: \bigoplus_{m < p} \Omega^m_{U_0'/S_0}[-m] \to \tau_{<p}\check\C^\bullet_U.$$
 Since $j_*\check\C^\bullet_U = \check\C^\bullet(\Y_0,F_*W^\bullet_{X_0/S_0})$, we obtain the desired homomorphism in $D^b(X_0')$ as $\psi^{-1} \circ j_*\varphi^\bullet$. It is a quasi-isomorphism because 
 $f_0: X_0 \to S_0$ has the Cartier isomorphism property by assumption.
\end{proof}

We like to apply this theorem to the case of a log toroidal family. It remains only to show that $\F rob(X,X')$ is a torsor:

\begin{proposition} \label{posCharSplit}
 In the above situation assume $f:X \to S$ log toroidal with respect to $S \to A_Q$. Then $\F rob(X,X')$ is a 
 $\G$-torsor, i.e., Frobenius liftings exist locally.
\end{proposition}
\begin{proof}
 Let $(Q\subset P,\F)$ be an ETD from a local model of $f: X \to S$, as given in \eqref{LM} with $S=\tilde S$. Consider the diagram 
\[
 \xymatrixcolsep{3.8em}
 \xymatrix{
  \, L \, \ar[d] \ar@{.>}[r]^{F} & \, L \, \ar[d] \ar[r]^c & A_{P,\F} \ar[d] \\
  \, S \, \ar[r]^{F_S} & \, S \, \ar[r]^a & \, A_Q. \\
 }
\]
We claim that for the local existence of a Frobenius lifting, it suffices to show that there is a scheme morphism $F: \underline L \to \underline L$ that is the underlying morphism of a log morphism on $c^{-1}(U_{P})$ such that the diagram commutes and the induced map 
$F \times_S S_0$ on $L_0 = L \times_S S_0$ is the absolute Frobenius.
Indeed, then $F$ plays the role of an absolute 
Frobenius on $L$, and its induced relative Frobenius gives rise to a local Frobenius lifting on $X'$ via the local model. 

The scheme $\underline L$ is affine with $\cO(L) = \bigoplus_{e \in E} z^e \cdot W_2(k)$
allowing us to define
$F: \underline L \to \underline L$ via $F^*(z^e \cdot w) := z^{pe} \cdot F_S^*(w)$.
It remains to extend $F$ to the log structure on $c^{-1}(U_{P})$.
Consider the maps of log schemes
$$M := \Spec (P \to \cO(L)) \to L \to \Spec(Q \to \cO(L)) =: N.$$
With the notation of Corollary \ref{smooth-type-decompo}, we define $W_i := c^{-1}(U_i)$. 
Observe that 
$M|_{W_1} = L|_{W_1}$ and $L|_{W_2} = N|_{W_2}$.
On $N$ and $M$, we get morphisms $F_N: N \to N$ and $F_M: M \to M$ by mapping $q \mapsto p\cdot q$ on the monoids and using $F^*$ on the rings. They are compatible with each other and with the maps to $S$,
and moreover $F_N \times_S S_0$ and $F_M \times_S S_0$ are the absolute Frobenii on $N_0, M_0$.
We define partially $F|_{W_1} := F_M|_{W_1}$ and $F|_{W_2} := F_N|_{W_2}$.
Because $N|_{W_1 \cap W_2} = L|_{W_1 \cap W_2} = M|_{W_1 \cap W_2}$ these definitions agree 
on $W_1 \cap W_2$
and we obtain a log morphism defined on $c^{-1}(U_{P}) = W_1 \cup W_2$ which gives the desired map.
\end{proof}

\section{The Hodge--de Rham Spectral Sequence}\label{TorAbsDeg}

We put the pieces together to prove Theorem \ref{absDegen} from the introduction. Let $S = \Spec (Q \to \kk)$ for a field $\kk \supset \QQ$ with $Q \ni q \mapsto \delta_{q0}$, and let $f: X \to S$ 
be a proper log toroidal family of relative dimension $d$ with respect to $S \to A_Q$. Setting $h^{pq} = \dime_\kk R^qf_*W^p_{X/S}$ and $h^n = \dime_\kk R^{n}f_*W^\bullet_{X/S}$, it suffices to prove $\sum_{p + q = n} h^{pq} = h^n$.

By Proposition \ref{spreadToroidal}, we can find an $S_\lambda = \Spec (Q \to B_\lambda)$ and a proper log toroidal family with respect to $S_\lambda \to A_Q$. Since $B_\lambda$ is integral, by shrinking $S_\lambda$, 
we can find a spreading out $\phi: \X \to \cS$ such that $R^q\phi_*W^p_{\X/\cS}$ and $R^n\phi_*W^\bullet_{\X/\cS}$ are locally free of constant rank $r^{pq}$ respective $r^n$ and such that $\cS/\ZZ$ is smooth as schemes. By Theorem 
\ref{baseChangeGeneric} we can furthermore assume that $W^m_{\X/\cS}$ is compatible with any base change, and we can assume that $\mathrm{char}\,\kappa(s) > d$ for the residue field $\kappa(s)$ of every closed point $s \in \cS$. Now let $\Spec k \to \cS$ be 
a closed point. Since $\cS/\ZZ$ is smooth, we can find a factorization 
$$\Spec k \to \Spec W_2(k) \to \cS$$
which induces diagram \eqref{SO} from the introduction by strict base change. Setting $g^{pq} := \dime_k R^q(\phi_k)_*W^p_{\X_k/k}$ and $g^{pq} := \dime_k R^n(\phi_k)_*W^\bullet_{\X_k/k}$, Lemma \ref{deriBaCha} yields 
$h^{pq} = r^{pq} = g^{pq}$ and $h^n = r^n = g^n$ hence it suffices to show $\sum_{p + q = n} g^{pq} = g^n$. Note that in diagram \eqref{SO} on the right, we are in the situation of Proposition \ref{posCharSplit}, so by Theorem~\ref{genDecompThm} we have a quasi-isomorphism 
$$\bigoplus_{m} W^m_{\X_k'/k}[-m] \simeq (F_0)_*W^\bullet_{\X_k/k}.$$
Now a computation as in \cite[Corollary~2.4]{Deligne1987} yields $\sum_{p + q = n} g^{pq} = g^n$ concluding the proof of Theorem \ref{absDegen}.

\subsection{The Relative Spectral Sequence}\label{sec-relative-degen} \emph{Proof of Theorem~\ref{rel-degen}}.
By Corollary~\ref{locBaChaField}, the formation of $W^p_{X/S}$ commutes with base change which is an ingredient for the classical base change theorem, e.g. \cite[\S3]{Deligne1968}, \cite[Theorem~(8.0)]{Katz70}. 
It thereby suffices to show the surjectivity of
$$\HH^k(X,W^\bullet_{X/S}) \to \HH^k(X_0,W^\bullet_{X_0/S_0}).$$ 
We prove this with the idea of \cite[Section (2.6)]{Steenbrink1976}, cf.~\cite[Lemma~4.1]{KawamataNamikawa1994} and \cite[Theorem~4.1]{GrossSiebertII}.
We define a complex 
$$\cL^\bullet := W^{\bullet,an}_{X}[u] =\bigoplus_{s = 0}^\infty W^{\bullet,an}_{X} \cdot u^s, 
\qquad d(\alpha_su^s) = d\alpha_s \cdot u^s + s\delta(\rho) \wedge \alpha_s \cdot u^{s - 1}$$
of analytic sheaves where $\rho = f^*(1) \in \M_{X^{an}}$ and 
$\delta: \M_{X^{an}} \to W^{1,an}_{X}$ is the log part of the universal 
derivation. Here $W^{\bullet,an}_{X}$ denotes (the analytification of) absolute differentials as in Corollary \ref{absolute-on-Y}.
Projection to the $u^0$-summand composed with $W^{\bullet,an}_{X} \to W^{\bullet,an}_{X/S}$  yields a map 
$\cL^\bullet \to W^{\bullet,an}_{X/S}$ whose composition with 
$W^{\bullet,an}_{X/S} \to W^{\bullet,an}_{X_0/S_0}$ fits into 
an exact sequence 
$$ 0 \to \K^\bullet \to \cL^\bullet \xrightarrow{\phi^\bullet} W^{\bullet,an}_{X_0/S_0} \to 0$$
of complexes that defines $\K^\bullet$. 
Since $f: X \to S$ has ETD local models, we may use Corollaries~\ref{relative-on-Y}, \ref{absolute-on-Y} and Remark~\ref{rem-compute-anal-stalk} to have a local description of this sequence.
Lemma~\ref{lem-K-acyclic} below shows that $\K^\bullet$ is acyclic for all ETDs with one-dimensional base, so $\phi^\bullet$ 
is a quasi-isomorphism and Theorem \ref{rel-degen} follows by the discussion in \S\ref{analytification}.
\hfill\qedsymbol

\begin{lemma} \label{lem-K-acyclic}
 Let $(\NN \subset P,\F)$ be an ETD, and let $f: X \to S = S_m$ be the base 
 change of $A_{P,\F} \to A_\NN$ along $S_m \to A_\NN$. With $0\in A_{P,\F}$ denoting the origin, we have
 $\cH^k(\K^\bullet)_0 = 0$ for all $k$.
\end{lemma}
\begin{proof}
 We choose Hermitian inner products on the vector spaces $L := P^{gp} \otimes \CC$ and $W := (P^{gp} \otimes \CC)/(\NN^{gp} \otimes \CC)$.
 With $K=(m+1)+\NN\subset\NN$, we recall $E_K$ from \S\ref{subsec-local-analytic}.
 For $e\in E_K$, we define
 $$L_e := \bigcap_{H \in \F_{\max} \setminus \F : e \in H} \hspace{-0.4cm} H^{gp} \otimes \CC \qquad \mathrm{and} \qquad 
 W_e :=  \bigcap_{H \in \F_{\max} \setminus \F : e \in H} \hspace{-0.4cm}
 (H^{gp} \otimes \CC)/(\NN^{gp} \otimes \CC).$$
 By Remark~\ref{rem-compute-anal-stalk} and Lemma~\ref{ana-stalk}, elements of $\cL^k_0$ are formal sums
 $$(\ell_{e,s}) := \sum_{s = 0}^N\sum_{e \in E_K} u^sz^e\ell_{e,s} \ , \quad
 \ell_{e,s} \in \bigwedge^kL_e \ , \quad 
 \mathrm{sup}_{\substack{e \in E_K \setminus 0\\ 1\le s\le N}} \left\{ \mathrm{log} \| \ell_{e,s}\| / h(e) \right\} < \infty ,$$ and elements of 
 $W^{k,an}_{X_0/S_0,0}$ are formal sums 
 $$(w_e) := \sum_{e \in E} z^e \cdot w_e, \quad
 w_e \in \bigwedge^kW_e \ , \quad
 \mathrm{sup}_{e \in E \setminus 0} \left\{ \mathrm{log} \| w_e \| / h(e) \right\} < \infty .$$
 Note that $(\ell_{e,s})$ is summed over $E_K$ whereas $(w_e)$ is summed 
 over $E$. 
 We denote the kernel of $\pi: \bigwedge^kL_e \to \bigwedge^kW_e$ by $K^k_e$ 
 and observe $\phi((\ell_{e,s})) = (\pi(\ell_{e,0}))$, so
 $(\ell_{e,s}) \in \K^k_0$ if and only if $\ell_{e,0} \in K^k_e$ for all $e \in E$.
 With $\bar\rho := 1 \otimes 1 \in \NN^{gp} \otimes \CC$ we have 
 $\delta(\rho) = z^0 \cdot \bar\rho \in W^1_{X}$ and thus 
 \begin{equation} \label{K-differential-explicit}
  d((\ell_{e,s})) = (e \wedge \ell_{e,s} + (s + 1)\bar\rho \wedge \ell_{e,s + 1}).
 \end{equation} 
 Let $(\ell_{e,s}) \in \K^0_0$ and assume $d((\ell_{e,s})) = 0$. 
 Since $\ell_{e,s} \in \CC$, for $e \not = 0$ by descending induction in $s$ starting 
 from $\ell_{e,N}$ we find $\ell_{e,s} = 0$. We have $\ell_{0,0} = 0$ 
 and ascending induction yields $\ell_{0,s} = 0$. Thus $\cH^0(\K^\bullet)_0 = 0$.
 
 Next, let $(\ell_{e,s}) \in \K^{k + 1}_0$ for $k \geq 0$ with $d((\ell_{e,s})) = 0$. 
 Starting with $e = 0$, we construct $(\tau_{e,s}) \in \K^k_0$ with $d((\tau_{e,s})) = (\ell_{e,s})$ using the following claim.
 
 \begin{claim} \label{claim1}
 Let $(L, \langle \cdot, \cdot\rangle)$ be a $\CC$-vector space of finite dimension with a Hermitian inner product. Let $0 \not= p \in L$ and $k \geq 0$, and assume $\ell \in \bigwedge^{k + 1}L$ with 
 $p \wedge \ell = 0$. Then there is a $\tilde \ell \in \bigwedge^k L$ with $p \wedge \tilde \ell = \ell$ and $\| p \| \cdot \| \tilde \ell \| = \| \ell \|$.
\end{claim}
\begin{proof}
 Let $\ell_1 := \frac{p}{\| p \|}, \ell_2, ..., \ell_n$ be an orthonormal basis of $L$, and $\{\ell_{i_1...i_k}\}$ the induced basis of $\bigwedge^k L$.
 If $\ell = \sum \alpha_{i_1...i_{k + 1}} \ell_{i_1...i_{k + 1}}$ satisfies the assumption, then
 $\tilde \ell = \frac{1}{\| p\|}\sum \alpha_{1i_2...i_{k + 1}}\ell_{i_2...i_{k + 1}}$ is a solution.
\end{proof}
We set $\tau_{0,0} = 0$. Writing out \eqref{K-differential-explicit} for $e=0$ yields
$$ d\big(\ell_{0,0}+\ell_{0,1}u+ \ell_{0,2}u^2+...\big)= \bar\rho \wedge \ell_{0,1}\ +\ 2\bar\rho \wedge \ell_{0,2}u\ +\ 3\bar\rho \wedge \ell_{0,3}u^2\ +\ ...$$
and therefore $\bar\rho \wedge \ell_{0,i} = 0$ for $i>0$. Since $\ell_{0,0} \in K^0_0$, we also have $\bar\rho \wedge \ell_{0,0} = 0$.
By Claim~\ref{claim1}, there is 
$\tau_{0,s+1} \in \bigwedge^kL_0$ 
with $\bar\rho\wedge \tau_{0,s + 1} = \ell_{0,s}$ and we are done with the case $e=0$. 
For $e \neq 0$ we need to care about convergence. 
Without loss of generality, $N \geq 1$. Since 
$e \wedge \ell_{e,N} = 0$, we can find by Claim~\ref{claim1}
$\tau_{e,N} \in \bigwedge^kL_e$ with $e \wedge \tau_{e,N} = \ell_{e,N}$ and 
$\|\tau_{e,N}\| \cdot \|e\| = \|\ell_{e,N}\|$. For $s \geq 1$, we construct 
$\tau_{e,s} \in \bigwedge^kL_e$ by descending induction. Because of 
$e \wedge (\ell_{e,s} - (s + 1)\bar\rho \wedge \tau_{e,s + 1}) = 0$, there is 
$\tau_{e,s}$ with 
$e \wedge \tau_{e,s} = \ell_{e,s} - (s + 1)\bar\rho \wedge \tau_{e,s + 1}$ 
and 
\begin{equation}\label{eq-norm-equality}
\|\tau_{e,s}\| \cdot \|e\| = \|\ell_{e,s} - (s + 1)\bar\rho \wedge \tau_{e,s + 1}\|. 
\end{equation}
For $e \notin E$, we go one step further and construct 
$\tau_{e,0} \in \bigwedge^kL_e$ with the same method, but for $e \in E$, the construction of
$\tau_{e,0}\in K^k_e$ is more intricate.
We need another claim:

 \begin{claim}\label{claim2}
 Let $(L,\langle \cdot, \cdot\rangle)$ be a $\CC$-vector space of finite dimension with a Hermitian inner product. 
 Let $0 \not= V, Y \subset L$ be subspaces with $V \cap Y = 0$. Then there is a 
 constant $\gamma>0$ with the following property: for every subspace $H$ with $V \subset H \subset L$ and $k \geq 0$, let $K^k_H$ be the kernel of $\bigwedge^k H \to \bigwedge^k(H/V)$.
 Then for every $0 \not= p \in Y \cap H$ and every $\ell \in K^{k + 1}_H$ with $p \wedge \ell = 0$, there is a $\tilde \ell \in K^k_H$ with $p \wedge \tilde \ell = \ell$ and $\gamma \cdot \| p \| \cdot \|\tilde \ell\| \leq \| \ell \|$.  
\end{claim}
\begin{proof}
Let $p=(p_1,p_2)$ be the decomposition of $p$ under $L=V\oplus V^\perp$, so $\|p\|^2=\|p_1\|^2+\|p_2\|^2$. Since $V\cap Y=0$, we have for
 $\gamma^2 := \inf_{0 \not= p \in Y}  \|p_2\|^2/\|p\|^2$ that $0<\gamma\le 1$.
Let  $\ell_0 := \frac{p_2}{\|p_2\|}, \ell_1,\ell_2...$ be an orthonormal basis of $H$ and then $\bar\ell_0=\frac{p}{\|p\|}$, $\bar\ell_i:=\ell_i$ for $i>0$ is an ordinary basis of $H$.
For $\ell = \sum \alpha_{i_0...i_k} \bar\ell_{i_0...i_k} \in K^{k + 1}_H$ with 
 $p \wedge \ell = 0$, we define 
 $\tilde \ell := \frac{1}{\|p\|}\sum \alpha_{0i_1...i_k} \bar\ell_{i_1...i_k} \in K^k_H$ 
 to have $p \wedge \tilde \ell = \ell$. We also find
 \begin{align}
  \|\ell\|^2 &= \left\|\sum \alpha_{0i_1...i_k}\frac{p}{\|p\|} \wedge \ell_{i_1...i_k} \right\|^2 \geq \left\|\sum \alpha_{0i_1...i_k}\frac{p_2}{\|p\|} \wedge \ell_{i_1...i_k} \right\|^2
 \geq \gamma^2 \cdot \|p\|^2 \cdot \|\tilde \ell\|^2 .\nonumber
 \end{align}
\end{proof}

We apply Claim~\ref{claim2} to $L = P^{gp} \otimes \CC$. Let $F_e \subset P$ be the face generated by $e$ and $Y = F_e^{gp} \otimes \CC$. Let 
$V = \NN^{gp} \otimes \CC$ and $H = L_e$, so $K^k_H = K^k_e$. 
Then $e \wedge (\ell_{e,0} - \bar\rho \wedge \tau_{e,1}) = 0$, so we find 
$\tau_{e,0} \in K^k_e$ with 
$e \wedge \tau_{e,0} = \ell_{e,0} - \bar\rho \wedge \tau_{e,1}$ and
\begin{equation}\label{eq-ind-start-triang-ineq}
\gamma\cdot\|\tau_{e,0}\|\cdot\|e\| \leq \|\ell_{e,0} - \bar\rho \wedge \tau_{e,1}\|.
\end{equation}
The factor $\gamma$ depends on $Y$, but there are only finitely 
many faces generated by elements $e \in E$, so we take for $\gamma$ the minimum over them and furthermore $\gamma<1$.
Applying the triangle inequality to the right hand side of \eqref{eq-ind-start-triang-ineq} and using induction and \eqref{eq-norm-equality} yields
$$\|\tau_{e,s}\| \leq \frac{1}{\gamma} \cdot \frac{1}{\|e\|} \sum_{k = s}^N 
\left(\frac{\|\bar\rho\|}{\|e\|}\right)^{k - s} \cdot \frac{k!}{s!} \cdot \|\ell_{e,k}\|$$
for all $e \not= 0$. Because $\inf_{e \not= 0}\{\|e\|\} > 0$, there is a bound 
$M > 1$ independent of $e$ such that $\|\tau_{e,s}\| \leq M \cdot \max_k\{\|\ell_{e,k}\|\}$ which proves 
$$\mathrm{sup}_{e \in E_K \setminus 0} \left\{ \mathrm{log} \| \tau_{e,s}\| / h(e) \right\} < \infty$$
and thus $(\tau_{e,s}) \in \K^k_0$. By construction, $d((\tau_{e,s})) = (\ell_{e,s})$, so $\cH^k(\K^\bullet)_0 = 0$.
\end{proof}

\section{Smoothings via Maurer--Cartan Solutions} \label{section-smoothing}
In the upcoming sections \S\ref{defo-from-MC} and \S\ref{MC-from-BV}, we adapt the methods of \cite{CLM} to the setup given in the statement of Theorem~\ref{maintheorem-tc}. 
We then argue how to obtain an analytic smoothing from a formal one in \S\ref{formal-to-analytic}.
The combination of all these sections gives a proof of Theorem~\ref{maintheorem-tc}. 
The main ingredients are Theorem~\ref{locally-unique-defos}, Theorem~\ref{absDegen} and Theorem~\ref{rel-degen}.
A key ingredient is also Lemma~\ref{lemma-add-E} to know that $W_{X/S}^d$ is trivial for $d={\dim X}$.

\subsection{Constructing a Formal Deformation from a Solution to the Maurer--Cartan Equation} \label{sec-maurer-Cartan}\label{defo-from-MC}
We define $\kS = \Spec(\NN \stackrel{1\mapsto t}{\longrightarrow} \CC[t]/t^{k + 1})$ and assume to be given a proper log toroidal family ${\oX}\ra {\oS}$. 
Let $\{{{}^0V_\alpha}\}_\alpha$ be an affine cover of ${\oX}$. 
For fixed $\alpha$, let $\{\kV\!_\alpha\ra \kS\}_k$ be a system of deformations, compatible with restriction from $k$ to $k-1$ as obtained from Theorem~\ref{locally-unique-defos}.
Note that $V_{\alpha\beta}:={}^0V_{\alpha}\cap {}^0V_{\beta}$ is affine because ${\oX}$ is separated. 
We give names to the restrictions of thickenings via $\kV\!_{\alpha;\alpha\beta}:={\kV\!_\alpha}|_{V\!_{\alpha\beta}}$.
Again by Theorem~\ref{locally-unique-defos}, we find isomorphisms
$${}^k\phi_{\alpha\beta}: \kV\!_{\alpha;\alpha\beta} \to \kV\!_{\beta;\alpha\beta}$$
of generically log smooth families over ${\kS}$ which are compatible with the restrictions to the base changes via ${{}^{k-1}\!S}\ra {\kS}$ but do not necessarily satisfy a cocycle condition.

We now analytify $\kX\ra\kS$ as well as $\kV\!_{\alpha},\kV\!_{\alpha;\alpha\beta}$. We keep using the same symbols though now refer to the analytifications respectively.

Let $\{U_i\}_{i\in I}$ be a cover of ${\oX}$ by Stein open sets that is also a basis for the analytic topology of ${\oX}$ with $I$ countable and totally ordered. Set $U_{i_0\dots i_l}:=\bigcap_{k=0}^l U_{i_k}$.
We obtain the sheaves of Gerstenhaber algebras
$${}^k\shG^p_\alpha:=\Theta^{-p}_{(\kV\!_{\alpha})/\kS}$$
concentrated in non-positive degrees
via the negative Schouten--Nijenhuis bracket $-[\cdot,\cdot]$ and $\wedge$. 
Set $\blacktriangle_l=\Spec(\CC[x_0,\dots,x_n]/(x_0+\dots+x_n-1))$ and $\shA^q(\blacktriangle_l)=\Omega^q_{\blacktriangle_l}$ and let $d_{j,l}:\blacktriangle_{l-1}\ra \blacktriangle_l$ be given by $x_j\mapsto 0$.
One constructs the Thom--Whitney bicomplex
\begin{equation}
\label{TW-def-eq}\tag{TW}
{^k}\!TW^{p,q}_{\alpha;\alpha_0\dots\alpha_l}=\left\{ (\varphi_{i_0\dots i_l})_{i_0<\dots<i_l}\,\left|\, 
\begin{array}{c}
U_{i_j}\subset V_{\alpha_0}\cap...\cap V_{\alpha_l} \hbox{ for }0\le j\le l,\\
\varphi_{i_0\dots i_l}\in\shA^q(\blacktriangle_l)\otimes_\CC{}^k\!\shG^p_\alpha(U_{i_0\dots i_l}), \\
d^*_{j,l}(\varphi_{i_0\dots i_l})=\varphi_{i_0\dots\hat i_j\dots i_l}|_{U_{i_0\dots i_l}}
\end{array}
\right.\right\}.
\end{equation}
The differential for the index $p$ is trivial and the differential $\bar\partial_\alpha$ for the index $q$ is induced by the de Rham differential on $\shA^q(\blacktriangle_l)$. 
Furthermore, $-[\cdot,\cdot]$ and $\wedge$ turn $TW$ into a Gerstenhaber algebra.
For $W\subset V_\alpha$, let ${^k}\!TW^{p,q}_{\alpha;\alpha}|_W$ be given by \eqref{TW-def-eq} but with the additional requirement to have $U_{i_j}\subset W$. 
The presheaf $W\mapsto {^k}\!TW^{p,\bullet}_{\alpha;\alpha}|_W$ gives a resolution of the sheaf ${}^k\shG^p_\alpha$ on $V_\alpha$, so ${}^k\shG^p_\alpha(W)=H^0_{\bar\partial_\alpha}({^k}\!TW^{p,\bullet}_{\alpha;\alpha}|_W)$.

The isomorphisms ${}^k\phi_{\alpha\beta}$ induce isomorphisms ${}^k\psi_{\alpha\beta}:{}^k\shG^\bullet_\alpha|_{V_{\alpha\beta}}\ra {}^k\shG^\bullet_\beta|_{V_{\alpha\beta}}$
of sheaves of Gerstenhaber algebras which can be used (\cite[Key Lemma~3.21]{CLM}) to construct isomorphisms
$$ {}^kg_{\alpha\beta}:{^k}\!TW^{p,q}_{\alpha;\alpha\beta}\ra {^k}\!TW^{p,q}_{\beta;\alpha\beta}$$
that satisfy the cocycle condition ${}^kg_{\gamma\alpha}{}^kg_{\beta\gamma}{}^kg_{\alpha\beta}=\id$ and are compatible with restriction from $k$ to $k-1$ and with $-[\cdot,\cdot]$ and $\wedge$.
The cocycle condition allows one to glue  $\{{^k}\!TW^{p,q}_{\alpha}\}_\alpha$ to a presheaf ${^k}\!\PV^{p,q}$ on ${\oX}$ compatible with restricting from $k$ to $k-1$. 
We set 
${^k}\!\PV^{n}:=\bigoplus_{p+q=n} {^k}\!\PV^{p,q}$.

While ${}^kg_{\alpha\beta}$ are not necessarily compatible with the differentials $\bar\partial_\alpha,\bar\partial_\beta$, there exist ${^k}\mathfrak{d}_\alpha\in {^k}\!TW^{-1,1}_{\alpha}$ such that
$(\bar\partial_\alpha+[{^k}\mathfrak{d}_\alpha,\cdot])_\alpha$ gives a system of maps compatible with ${}^kg_{\alpha\beta}$ (\cite[Theorem~3.34]{CLM}). This system glues to an operator $\bar\partial$ on ${^k}\!\PV^{p,q}$ compatible with restriction from $k$ to $k-1$.
However, $\bar\partial$ is not a differential because
$$\bar\partial^2 = \left[{^k}\mathfrak{l}_\alpha,\,\cdot\,\right]\quad\hbox{ for }\quad {^k}\mathfrak{l}_\alpha:=\bar\partial_\alpha({^k}\mathfrak{d}_\alpha)+\frac12[{^k}\mathfrak{d}_\alpha,{^k}\mathfrak{d}_\alpha]\in{^k}\!TW_\alpha^{-1,2}.$$
The $\{{^k}\mathfrak{l}_\alpha\}_\alpha$ glue to a global element ${^k}\mathfrak{l}\in {^k}\!\PV^{-1,2}$ that is compatible with restricting from $k$ to $k-1$. If ${^k}\phi\in {^k}\!\PV^{-1,1}$ solves the Maurer--Cartan equation
\begin{equation}\label{MC1}\tag{MC1}
\bar\partial({^k}\phi)+\frac12[{^k}\phi,{^k}\phi]+{^k}\mathfrak{l}=0,
\end{equation}
then $(\bar\partial+[{^k}\phi,\cdot])^2=0$. 
In this case the cohomology $H^\bullet_{(\bar\partial+[{^k}\phi,\cdot])}({^k}\!\PV^\bullet)$ is a presheaf of Gerstenhaber algebras on $\oX$ that is locally isomorphic to ${}^k\shG^\bullet_\alpha$.
The sheafification of its degree zero part gives a sheaf $\shO_{X_k}$ of $\CC[t]/t^{k+1}$-algebras on $\oX$ which we take as the $k$th order deformation of $\oX$. Taking the limit $\shO_{\mathfrak{X}}:=\liminv_{k}\shO_{X_k}$ yields a flat and proper morphism $\mathfrak{X}\ra \mathfrak{S}$ with $\mathfrak{S}:=\operatorname{Spf}(\CC\lfor t\rfor)$.

\subsection{Constructing a Solution to the Maurer--Cartan Equation using the Batalin--Vilkovisky Operator}\label{MC-from-BV}
We assume that $W^d_{\oX/\oS} \cong \cO_\oX$.
We fix a global generator ${}^0\omega\in\Gamma(\oX,W^d_{\oX/\kS})$. 
Let ${}^k\omega_\alpha\in \Gamma({}^0V_\alpha,W^d_{\kV\!_\alpha/\kS})$ be a choice of generator that is a lift to $k$ of ${}^0\omega|_{{}^0V_\alpha}$.
The Batalin--Vilkovisky operator ${}^k\Delta_\alpha$ is the transfer of the de Rham differential $\texttt{d}$ to the polyvector fields, i.e., ${}^k\Delta_\alpha$ is the composition
$$ \Theta^{p}_{(\kV\!_{\alpha})/\kS}\stackrel{\invneg({}^k\omega_\alpha)}{\lra} W^{d-p}_{(\kV\!_{\alpha})/\kS}\stackrel{\texttt{d}}{\lra} W^{d-p+1}_{(\kV\!_{\alpha})/\kS}\stackrel{\invneg({}^k\omega_\alpha)^{-1}}{\lra} \Theta^{p-1}_{(\kV\!_{\alpha})/\kS}$$
and thus a differential ${}^k\shG^p_\alpha\ra {}^k\shG^{p+1}_\alpha$. Choosing ${}^k\omega_\alpha$ compatible with restricting from $k$ to $k-1$, also the ${}^k\Delta_\alpha$ share this property.
For $W\subset {}^0V_\alpha\cap {}^0V_\beta$ there is $\lambda_{\alpha\beta}\in \Gamma(W,{}^k\shG^0_\alpha)$ with ${}^k\omega_\alpha|_W=\lambda_{\alpha\beta}\cdot{}^k\omega_\beta|_W$. 
Setting ${^k}\mathfrak{w_{\alpha\beta}}:=\log(\lambda_{\alpha\beta})$ yields
$$ {}^k\psi_{\beta\alpha}\circ {}^k\Delta_\beta\circ {}^k\psi_{\alpha\beta} - {}^k\Delta_\alpha=[{^k}\mathfrak{w}_{\alpha\beta},\,\cdot\,],$$
and then $\{{^k}\mathfrak{w_{\alpha\beta}}\}_{\alpha\beta}$ can be upgraded (\cite[Theorem~3.34]{CLM}) to a \v{C}ech cocycle for ${^k}\!TW^{0,0}_{\alpha;\alpha\beta}$ which by exactness lifts to a collection ${^k}\mathfrak{f}_\alpha\in TW^{0,0}_{\alpha}$. The collection is compatible with restricting from $k$ to $k-1$ and satisfies
$$ {}^kg_{\beta\alpha}\circ ({}^k\Delta_\beta+[{}^k\mathfrak{f}_\beta,\cdot])\circ {}^kg_{\alpha\beta} =({}^k\Delta_\alpha+[{}^k\mathfrak{f}_\alpha,\cdot]).$$
Since ${}^k\mathfrak{f}_\alpha$ lives in degree $(0,0)$, one has $({}^k\Delta_\alpha+[{}^k\mathfrak{f}_\alpha,\cdot])^2=0$, so we can glue the collection $\{{}^k\Delta_\alpha+[{}^k\mathfrak{f}_\alpha,\cdot]\}_\alpha$ to an operator 
$\Delta:{}^k\PV^{p,q}\ra{}^k\PV^{p+1,q}$ with $\Delta^2=0$. Now,
$$ \Delta \bar\partial + \bar\partial\Delta = [{}^k\mathfrak{y},\cdot]\quad \hbox{ for } \quad {}^k\mathfrak{y}_\alpha:={}^k\Delta_\alpha({}^k\mathfrak{d}_\alpha)+{}^k\bar\partial_\alpha({}^k\mathfrak{f}_\alpha)+[{}^k\mathfrak{d}_\alpha,{}^k\mathfrak{f}_\alpha]$$
and ${}^k\mathfrak{y}\in {}^k\!\PV^{0,1}$ is glued from the collection ${}^k\mathfrak{y}_\alpha$. 
By construction,
$$\breve{d}:=\bar\partial+\Delta+(\mathfrak{l}+\mathfrak{y})\wedge$$
satisfies $\breve{d}^2=0$ and furthermore $(\mathfrak{l}+\mathfrak{y})\equiv 0\mod (t)$. 
\begin{theorem} \label{surjective-PV}
The natural maps $H^i_{\!\breve{d}}({}^k\!\PV^\bullet)\ra H^i_{\!\breve{d}}({}^{k-1}\!\PV^\bullet)$ are surjective for all $i$ and $k$.
\end{theorem}
\begin{proof} 
As in \cite[Proposition~4.8]{CLM}, the elements $\exp({^k}\mathfrak{f}_\alpha\invneg){^k}\omega_\alpha$ glue to a global element ${^k}\omega$ in the Thom--Whitney de Rham complex $({}^k_\parallel\!\shA^{\bullet},d)$ (constructed from $W^\bullet_{\kV\!_\alpha/\kS}$ in our case) compatible with restricting from $k$ to $k-1$. 
Contracting ${^k}\omega$ gives an isomorphism of complexes ${}^k\!\PV^\bullet\ra {}^k_\parallel\!\shA^{\bullet}$, so it suffices to prove surjectivity of $H^i_{\!{d}}({}^k_\parallel\!\shA^{\bullet})\ra H^i_{\!{d}}({}^{k-1}_{\ \ \,\,\parallel}\!\shA^{\bullet})$. This follows from Theorem~\ref{rel-degen}, cf.~\cite[Lemma~4.17]{CLM}. 
\end{proof} 
\begin{remark}
For a formal variable $u^{\frac12}$, consider on $\PV^\bullet\lfor u^{\frac12}\rfor$ the differential $\breve{d}_u:=\bar\partial+u\Delta+u^{-1}(\mathfrak{l}+u\mathfrak{y})\wedge$. A direct computation gives
$\breve{d}_u=u^{\frac12} I_u^{-1} \circ\breve{d} \circ I_u$ where $I_u$ is defined by $I_u(\varphi) = u^{\frac{p-q-2}2}\varphi$ for $\varphi\in \PV^{p,q}\lfor u^{\frac12}\rfor[u^{-\frac12}]$ (cf.~\cite[Notation~5.1]{CLM}).
Theorem~\ref{surjective-PV} thus implies that
\begin{equation}\label{surjectiveu12}
H^i_{\!\breve{d}_u}({}^k\!\PV^\bullet\lfor u^{\frac12}\rfor[u^{-\frac12}])\ra H^i_{\!\breve{d}_u}({}^{k-1}\!\PV^\bullet\lfor u^{\frac12}\rfor[u^{-\frac12}])
\end{equation}
is surjective for all $i,k$. 
\end{remark}
\begin{theorem} \label{freeness-PV} For all $i$,
$H^i_{\bar\partial+u\Delta}({}^0\!\PV^\bullet\lfor u\rfor)$ is a free $\CC\lfor u\rfor$-module of finite rank.
\end{theorem} 
\begin{proof} Note that $k=0$.
With $\bar\partial$ the \v{C}ech differential for the cover $\{V_\alpha\}_\alpha$, 
the degeneration of the Hodge--de Rham spectral sequence for $(W^\bullet_{\oX/\oS},d)$ at $E_1$ by Theorem~\ref{absDegen} is equivalent to 
$H^i_{\bar\partial+ud}(\{V_\alpha\}_\alpha,W^\bullet_{\oX/\oS}\lfor u\rfor)$ being a free $\CC\lfor u\rfor$-module of finite rank.
The quasi-isomorphisms $W^\bullet_{\oX/\oS}\lfor u\rfor\ra {}^0_\parallel\!\shA^{\bullet}\lfor u\rfor$ and ${}^0\!\PV^\bullet\lfor u\rfor\ra {}^0_\parallel\!\shA^{\bullet}\lfor u\rfor$ yield the assertion.
\end{proof}

\begin{theorem} \label{main-theorem-MC}
There exist ${}^k\varphi\in {}^k\!\PV^0\lfor u\rfor$ for all $k\ge 0$ with ${}^k\varphi\equiv {}^{k+1}\varphi \mod t^{k+1}$ and ${}^0\varphi=0$ solving
\begin{equation}
\label{MC2}\tag{MC2}
(\bar\partial+u\Delta)({^k}\varphi)+\frac12[{^k}\varphi,{^k}\varphi]+({^k}\mathfrak{l}+u\,{^k}\mathfrak{y})=0.
\end{equation}
Furthermore, setting ${^k}\phi:=({^k}\varphi \mod u)$ with ${^k}\phi=\sum_j {^k}\phi_j$ and ${^k}\phi_j\in {}^k\!\PV^{-j,j}$, it holds ${^k}\phi_0=0$ and thus ${^k}\phi_1\in {}^k\!\PV^{-1,1}$ solves \eqref{MC1}.
\end{theorem}
\begin{proof} 
The first assertion becomes \cite[Theorem~5.5]{CLM} if we set $\mathbf{I}=(t)$ and $\psi=0$ and check that we have the ingredients for its proof available. The proof goes by induction over $k$ and uses
(i) the surjectivity in Theorem~\ref{surjective-PV} for $k=0$, (ii) the surjectivity in Equation~\eqref{surjectiveu12} for all $k$ and (iii) Theorem~\ref{freeness-PV} in each step to get rid of negative powers of $u$ in ${^k}\varphi$.
The second statement is \cite[Lemma~5.11]{CLM}.
\end{proof}

\subsection{From a Formal Deformation to an Analytic Deformation} \label{formal-to-analytic}
Let $\mathfrak{S}$ be the completion of an analytic variety $S$ in a non-zero divisor $t\in \Gamma(S,\shO_S)$. 
Let $S_k$ be the closed analytic subvariety defined by $t^k$. If $X\ra S$ is flat, we denote by $X_k\ra S_k$ the base change to $S_k$, similarly for a flat map $\mathfrak{X}\ra \mathfrak{S}$.
\begin{theorem}[\cite{RS19}, Theorem B.1] 
\label{thm-approx-main}
Given a proper and flat formal analytic morphism $\hat\varphi:\mathfrak{X}\ra \mathfrak{S}$, for every $k>0$ there is a proper flat analytic morphism
$\varphi:X\ra S$ together with an $S_k$-isomorphism $\mathfrak{X}_k\ra X_k$ of the base changes of $\hat\varphi$ and $\varphi$ to $S_k$. 
\end{theorem} 

\begin{theorem}[\cite{Ru18}, Theorem~5.5\,(1)]
\label{thm-approx-nbd}
In the situation of Theorem~\ref{thm-approx-main}, given $s\in S_0$ and $X_s=\varphi^{-1}(s)$, there exists an integer $K>0$ such that whenever $\varphi:X\ra S$ is obtained for $k>K$ then every point $x\in X_s$ has a neighborhood in $X$ whose $t$-completion is formally isomorphic to a neighborhood of $x$ in $\mathfrak{X}$, in particular if $\mathfrak{X}$ is a smoothing of a fiber $X_s$ for $t\neq 0$ then so is $X$. 
\end{theorem}

\begin{theorem}[\cite{Ru18}, Theorem~5.5\,(3)] 
\label{thm-approx-log}
In the situation of Theorem~\ref{thm-approx-nbd}, for $X_0$ the base change to $S_0$, the maps of pairs $(X,X_0)\ra (S,S_0)$ and $(\mathfrak{X},\mathfrak{X}_0)\ra (\mathfrak{S},S_0)$ turn $\hat\varphi$ and $\varphi$ into log morphisms via the divisorial log structures. There is an isomorphism of the log fibers over $s\in S$ whose underlying morphism is the restriction to the fiber $X_s$ of the $S_k$-isomorphism $\mathfrak{X}_k\ra X_k$.
\end{theorem}

{\footnotesize
\bibliography{smoothy-pi}
\bibliographystyle{plain}
}

\end{document}